\documentclass[12pt]{article}
\usepackage{amssymb, amsmath, hyperref}
\usepackage{amsthm}
\usepackage{bm}
\usepackage[margin=0.8in]{geometry}
\usepackage{enumerate}
\allowdisplaybreaks[1]
\numberwithin{equation}{section}

\newtheorem{notation}{Notation}[section]
\newtheorem{example}{Example}[section]
\newtheorem{thm}{Theorem}[section]
\newtheorem{cor}{Corollary}[section]
\newtheorem{note}{Note}[section]
\newtheorem{pro}{Proposition}[section]

\newtheorem{defn}{Definition}[section]

\newtheorem{lemma}{Lemma}[section]
\usepackage{graphicx}

\begin{document}
	\vspace{-2cm}
	\markboth{R.Vishnupriya and R. Rajkumar}{New matrices for spectral hypergraph theory, II}
	\title{\LARGE\bf 
		New matrices for spectral hypergraph theory, II}
	\author{R. Vishnupriya\footnote{e-mail: {\tt rrvmaths@gmail.com},} \footnote{First author is supported by University Grands Commission (U.G.C.), Government of India under the Fellowship Grant No. 221610053976,},~ 
		R. Rajkumar\footnote{e-mail: {\tt rrajmaths@yahoo.co.in}.}\\ 
		{\footnotesize Department of Mathematics, The Gandhigram Rural Institute (Deemed to be University),}\\ \footnotesize{Gandhigram -- 624 302, Tamil Nadu, India}\\[3mm]
	}
	\date{}
	\maketitle
	\begin{abstract}
	The properties of a hypergraph explored through the spectrum of its unified matrix  was made by the authors in~\cite{vishnu part1}. In this paper, we introduce three  different hypergraph matrices: unified Laplacian matrix, unified signless Laplacian matrix, and unified normalized Laplacian matrix, all defined using the unified matrix. We show that  these three matrices of a hypergraph are respectively identical to the  Laplacian matrix,  signless Laplacian matrix, and normalized Laplacian matrix of the associated graph. This allows us to use the spectra of these hypergraph matrices as a means to connect the structural properties of the hypergraph with those of the associated graph. Additionally, we introduce certain hypergraph structures and invariants during this process, and relate them to the eigenvalues of these three matrices.
		\vspace{-0.3cm}
		
	
\end{abstract}

\section{Introduction}

Spectral graph theory examines graph properties through the eigenvalues and eigenvectors of matrices like the adjacency matrix, Laplacian matrix, signless Laplacian matrix, and normalized Laplacian matrix (see,~\cite{bapat, chung bk, cvetko}). A hypergraph extends the concept of a graph by allowing edges, or hyperedges, to connect more than two vertices. In spectral hypergraph theory, both matrices and tensors have been associated with hypergraphs in recent decades, with tensors being introduced in~\cite{Bulo, cooper2012spectra, Hu, Li, Xie, Qi, Banarjee tensor}. For more on tensor spectra associated with hypergraphs, see ~\cite{Qi tensor bk}. However, computing eigenvalues is an NP-hard problem, and not all aspects of spectral graph theory extend smoothly to hypergraphs when using tensors. These challenges highlight the limitations of tensor-based methods in spectral hypergraph theory. Matrix-based approaches, as seen in ~\cite{ad mat 1, Rod 2002, reff2012, banerjee2021spectrum, cardoso}, address these issues by analyzing hypergraph structures through the spectra of associated matrices. A graph can be uniquely determined by its adjacency matrix, and similarly, a hypergraph can be identified by its associated tensors mentioned above. However, this uniqueness may not hold for the associated matrices mentioned above.

To address this, in~\cite{vishnu part1}, we introduced the unified matrix for a hypergraph, which is identical to the adjacency matrix of the associated graph. We used the unified matrix's spectrum to link the structural properties of the hypergraph with those of the graph. We also introduced certain hypergraph structures and invariants, related them to the eigenvalues of the unified matrix.

This paper we introduce three distinct hypergraph matrices: the unified Laplacian matrix, the unified signless Laplacian matrix, and the unified normalized Laplacian matrix, all defined using the unified matrix. We show that these matrices are identical to the Laplacian, signless Laplacian, and normalized Laplacian matrices of the associated graph. This allows us to study the spectra of these hypergraph matrices to establish connections between the structural properties of hypergraphs and their associated graphs.

As the three matrices mentioned above are defined using the unified matrix, we will be referred~\cite{vishnu part1} in
the sequel as Part~I. While we adhere to the terminology and notation from Part~I, some of it will be repeated here for the convenience of the reader.

The rest of the paper is arranged as follows. Section~\ref{Lsec2} contains some basic definitions and notations on matrices, graphs and hypergraphs. In addition, we recall some definitions introduced in Part~I. In Section~\ref{Lsec4}, we introduce the 
unified Laplacian matrix of a hypergraph.  
We bound the largest unified Laplacian eigenvalue and the second smallest unified Laplacian eigenvalue, namely the algebraic $d$-connectivity of a hypergraph using some hypergraph invariants.
In addition, we introduce different types of connectedness, distances and diameters in a hypergraph. Further, we establish the relationship between those connectedness and diameters of a hypergraph.  Also, we bound the diameters of a hypergraph using its unified Laplacian eigenvalues and some hypergraph invariants.
Moreover, we extend the Matrix-Tree Theorem for a graph to hypergraphs. In Section~\ref{Lsec5}, we introduce the unified signless Laplacian matrix of a hypergraph. We characterize deeply connected hypergraphs having no odd exact cycle using their signless Laplacian eigenvalue zero and using its arithmetic multiplicity, we analyze some hypergraph properties. Also, we bound its largest unified singless Laplacian eigenvalue using some hypergraph invariants. In Section~\ref{Lsec6}, we define the unified normalized Laplacian matrix of a hypergraph. We discuss when this matrix has two as its eigenvalue, and we count its arithmetic multiplicity using some properties of a hypergraph. We bound unified normalized Laplacian eigenvalues using exact set diameter and some hypergraph invariants. We introduce the unified Cheeger constant of a hypergraph and bound it by using the second smallest unified normalized Laplacian eigenvalue of a hypergraph. Also, we introduce exact set distance between two subsets of $I(H)$ and provide some spectral bounds on it. In Section~\ref{Lsec7}, present some facts about cospectral hypergraphs with respect to these four matrices


\section{Preliminaries}\label{Lsec2}
Let $S$ be a non-empty set. We denote the set of all non-empty subsets of $S$ by $\mathcal{P}^*(S)$.
A \textit{hypergraph} $H$ consists of a non-empty set $V(H)$ and a multiset $E(H)$ of non-empty subsets of $V(H)$.
The elements of $V(H)$ are called \textit{vertices} and the elements of $E(H)$ are called \textit{hyperedges}, or simply \textit{edges} of $H$. We denote the set of all elements in the multiset $E(H)$ by $E^*(H)$. $H$ is said to be \textit{$m$-uniform} if all of it’s edges have the cardinality $m$.
An edge $e$ in $H$ is called \textit{included} if there exist an edge $e'(\neq e)$in H such that $e\subset e'$.
The \textit{multiplicity} of an edge $e$ in $H$, denoted by $m(e)$, is the number of occurrences of that edge in $H$. An edge $e$ of $H$ is called \textit{multiple} if $m(e)\geq2$.  An edge of $H$ having cordinality one is called a \textit{loop}. The \textit{degree} of a vertex $v$ in $H$, denoted by $d_H(v)$, is the number of edges in $H$ containing $v$. 
Let $\delta(H)$ and $\Delta(H)$ denote the minimum and the maximum degrees of all the vertices in $H$, respectively. 
$H$ is said to be \textit{trivial} if it has only one vertex and no edge. 
$H$ is said to be \textit{simple} if it has no loops or multiple edges.  A \textit{subhypergraph} of $H$ is a hypergraph $H'$ with $V(H')\subseteq V(H)$ and $E(H')\subseteq E(H)$.

Now, we state some definitions and notations on graphs: Let $G$ be a graph. The \textit{distance} between two vertices $u$ and $v$ in $G$, denoted by $d_G(u,v)$, is the length of the shortest path joining $u$ and $v$. If there is no path joining $u$ and $v$ in $G$, then $d_G(u,v)$ is defined to be $\infty$. In a connected graph $G$, the value, $\max\{d_G(u,v)~|~u,v\in V(G)\}$ is called the \textit{diameter} of $G$. 		
Let $X$, $Y$ be non-empty subsets of $V(G)$. The \textit{distance between $X$ and $Y$} is the minimum distance between a vertex in $X$ and a vertex in $Y$. Let $E(X, Y)$ denotes the set of all edges in $G$ with one vertex in $X$ and the other in $Y$.  The \textit{volume} of a subset $S$ of $V(G)$ is defined as the sum of the degrees of all the vertices in $S$ and is denoted by $vol(S)$. We denote the $vol(V(G))$ simply by $vol(G)$. The Cheeger constant of $G$ is denoted by $h(G)$ and is defined by $h(G)=\underset{\Phi\subset S\subset V(G)}{\min}h_G(S)$, where $h_G(S)=\displaystyle\frac{|E(S, V(G)\backslash S)|}{\min \{vol(S), vol(V(G)\backslash S)\}}$. $P_n$, $C_n$ and $K_n$ denote the path, the cycle and the complete graph on $n$ vertices, respectively. The complete $r$-partite graph with partitions of cardinality $n_1,n_2,\dots, n_r$ is denoted by $K_{n_1,n_2,\dots,n_r}$.

 Let $G$ be a finite, loopless graph with vertex set $V(G)= \{v_1,v_2,\dots,v_n\}$ and edge set $E(G) = \{ e_1,e_2,\dots,e_m \}$. 
The \textit{adjacency matrix} of $G$, denoted by $A(G)$, is the matrix of order $n$ whose rows and columns are indexed by the vertices of $G$ and for all $v_i, v_j\in V(G)$,
\begin{center}
	the $(v_i, v_j)^{th}$ entry of $A(G)=$
	$\begin{cases}
		m(\{v_i,v_j\}), &\text{if}~i\neq j~\text{and}~\{v_i,v_j\}\in E(G);\\
		0,&\text{otherwise}.
	\end{cases}$ 
\end{center}
The \textit{degree matrix} of $G$, denoted by $D(G)$, is the diagonal matrix of order $n$ whose rows and columns are indexed by the vertices of $G$ and the $(v_i,v_i)$-th entry is the degree of the vertex $v_i$ in $G$ for $i=1,\dots,n$. 
The \textit{Laplacian matrix} of $G$, denoted by $L(G)$, is the matrix $D(G)-A(G)$. The \textit{signless Laplacian matrix} of $G$, denoted by $Q(G)$, is the matrix $D(G)+A(G)$.
The \textit{normalized Laplacian matrix} of $G$, denoted by $\hat{L}(G)$, is the matrix of order $n$ whose rows and columns are indexed by the vertices of $G$ and for all $v_i, v_j\in V(G)$,  \begin{center}
	the $(v_i, v_j)^{th}$ entry of $\hat{L}(G)=$
	$\begin{cases}
		1, &\text{if}~i= j~\text{and}~d_G(v_i)\neq0;\\
		-\frac{m(\{v_i,v_j\})}{\sqrt{d_G(v_i)d_G(v_j)}}, &\text{if}~i\neq j~\text{and}~\{v_i,v_j\}\in E(G);\\
		0,&\text{otherwise}.
	\end{cases}$ 
\end{center}
Notice that if $H$ has no isolated vertices, then $\hat{L}(G)=D(G)^{-\frac{1}{2}}L(G)D(G)^{-\frac{1}{2}}$.
The \textit{incidence matrix} of a simple graph $G$, denoted by $B(G)$, is the $0-1$ matrix whose rows and columns are indexed by the vertices and the edges of $G$, respectively. The $(v_i, e_j)-$th entry of $B(G)$ is $1$ if and only if $v_i\in e_j$.
Consider an orientation to each edge of a simple graph $G$. Then the \textit{vertex-arc incidence matrix} of $G$, denoted by $R(G)$, is the matrix whose rows and columns are indexed by the vertices and the arcs of $G$, respectively.

The $(v_i, e_j)-$th entry of $R(G)$=$\begin{cases}
	1,& \text{if}~v_i~ \text{is the head of}~e_j;\\
	-1,&	\text{if}~v_i~ \text{is the tail of}~e_j;\\
	0,& \text{otherwise}.
\end{cases}$

We denote the all-ones matrix of size $n\times m$ by $J_{n\times m}$.  We denote $\mathbf{0}$ as the zero matrix of appropriate size. 	A matrix is said to be \textit{totally unimodular} if the determinant of any of its square submatrix is either $0$ or $\pm1$. 
The rank of a matrix $M$ is denoted by $r(M)$. If $M$ is a square matrix, then the trace of $M$ is denoted by $tr(M)$, and the characteristic polynomial of $M$ is denoted by $P_M(x)$. The eigenvalues of $M$ are denoted by $\lambda_i(M)$, $i=1,2,\dots,n$. The spectrum of $M$ is the multiset of eigenvalues of $M$ and is denoted by $\sigma(M)$. 
\subsection*{{\it 2.1 Definitions introduced in Part~I}}\label{Lsec3}

Now, we recall some notations and definitions introduced in Part~I that we will use in this paper. 
Let $S$ be a non-empty set. If $\{S_1,S_2\}$ is a $2$-partition of $S$, then we call $S_1$ and $S_2$ as parts of $S$. Let $\tau(S)$ denote the set of all $2-$partitions of $S$.

Let $H$ be a hypergraph. Let $I(H)$ denote the set of all the parts of each edge of $H$ together with all the singleton subsets of $V(H)$. The cardinality of $I(H)$ is called the \textit{edge index} or simply the \textit{$e$-index} of $H$. 		
For $S$ and $S'\in I(H)$,  \textit{$S$ is said to be a neighbor of $S'$ with multiplicity} $c~(\geq1)$, if $S$ and $S'$ forms a partition of $c$ number of edges in $H$. It is denoted by $S\overset{c}{\sim}S'$. Also, we simply write $S\sim S'$, whenever there is no necessity to mention the multiplicity of $S$ explicitly.

Let $H$ be a hypergraph with $e$-index $k$. Let $I(H)=\{S_1, S_2,\dots, S_{k}\}$. Then \textit{the unified matrix of $H$}, denoted by $\mathbf{U}(H)$, is the matrix of order $k$ whose rows and columns are indexed by the elements of $I(H)$ and for all $S_i, S_j\in I(H)$, 
\begin{center}
	the $(S_i, S_j)^{th}$ entry of $\mathbf{U}(H)=$
	$\begin{cases}
		m(\{S_i\}), &\text{if}~i=j,~|S_i|=1~\text{and}~S_i\in E(H);\\
		c, &\text{if}~S_i \overset{c}{\sim} S_j;\\
		0,&\text{otherwise}.
	\end{cases}$
\end{center}

The \textit{associated graph} of a loopless hypergraph $H$, denoted by $G_H$,
is the graph  whose vertex set is $I(H)$ and two vertices $S$ and $S'$ are joined by $c~(\geq1)$ number of edges in $G_H$ if and only if $S \overset{c}{\sim} S'$ in $H$.  Notice that $A(G_H)=\mathbf{U}(H)$ and so they have the same spectrum. 
Also, if $H$ is a loopless graph, then the eigenvalues of $A(H)$ and the eigenvalues of $\mathbf{U}(H)$ are the same. Thereby, we denote the eigenvalues of these two matrices commonly as $\lambda_i(H)$, $i=1,2,\dots,k$.

Let $H$ be a hypergraph. For each $S\in I(H)$, the \textit{unified degree of $S$ in $H$}, denoted by $d_H(S)$, is the cordinality of the multiset $\{e\in E(H)~|~S\subseteq e\}$.

An \textit{exact walk} in $H$ is a sequence $EW:(S=)S_0,$ $e_1,S_1,e_2,S_2,$ $\dots,e_{n-1},$ $S_{n-1},e_n,S_n(=S')$, where $\{S_{i-1}, S_i\}$ is a $2$-partition of the edge $e_i$ in $H$ for $i=1,2,3,\dots,n$. We say $EW$ joins $S$ and $S'$. Each $S_i$ occurs in $EW$ is a \textit{part} of $EW$. $S$ and $S'$ are \textit{the initial part} and \textit{the terminal part} of $EW$, respectively; Other parts are called internal parts of $EW$. We call the set of all parts of $EW$ as its \textit{cover}. If $S=\{u\}$ and $S'=\{v\}$, then we say that $EW$ joins $u$ and $v$. The \textit{length of $EW$}, denoted by $l(EW)$, is the number of edges in $EW$, i.e., $n$.  
We denote the exact walk simply by $EW:(S=)S_0,S_1,S_2,\dots,S_{n-1},S_n(=S')$, when the edges of the walk are self evident.

An exact walk in a hypergraph $H$ is said to be an \textit{exact path} if all its parts are distinct.	
An exact walk in $H$ with at least three distinct edges is said to be an \textit{exact cycle} if its initial and terminal parts are the same and all the internal parts are distinct. We say an exact cycle as \textit{odd (even)} if its length is odd (even).  A hypergraph with at least two (three) distinct edges is said to be an \textit{unified path} (\textit{unified cycle}) if its edges and some of their parts can be arranged in an exact path (exact cycle) sequence whose parts are pairwise disjoint.

A hypergraph $H$ is said to be \textit{exactly connected} if for any two distinct vertices $u$ and $v$ in $H$, there exists $S, S'\in I(H)$ with $u\in S$ and $v\in S'$ such that $S$ and $S'$ are joined by an exact path in $H$. The \textit{exact distance} or simply the \textit{$e$-distance} between the distinct vertices $u$ and $v$ in $H$, denoted by $ed_H(u,v)$, is the length of a shortest exact path joining $S, S'\in I(H)$ with $u\in S$ and $v\in S'$, if such an exact path exists; $\infty$, otherwise. For a vertex $u$ in $H$, we define $ed_H(u,u)=0.$ 
The \textit{exact diameter} of an exactly connected hypergraph $H$,  denoted by $ED(H)$, is the maximum of the $e$-distance between all the pairs of vertices in $H$, i.e., $ED(H)=\max\{ed_H(u,v)~|~u,v\in V(H)\}$.


\section{Unified Laplacian matrix of a hypergraph}\label{Lsec4}
In this section, we introduce the 
unified Laplacian matrix of a hypergraph and relate some of its invariants to the hypergraph invariants. 
\begin{defn}\normalfont
	Let $H$ be a hypergraph. For $S\in I(H)$, we define the \textit{modified unified degree of $S$ in $H$}, denote by 	$d^*_H(S)$, as
	
	$d^*_H(S)=\begin{cases}
	d_H(S),&\text{if}~|S|=1~\text{or}~S\notin E(H);\\
	d_H(S)-m(S),&\text{otherwise.}
	\end{cases}$ 
\end{defn}
\begin{defn}\normalfont
Let $H$ be a hypergraph. For a non-empty subset $\mathcal{S}$ of $I(H)$, we define the \textit{volume of $\mathcal{S}$} in $H$, denoted by $vol_H(\mathcal{S})$, as $$vol_H(\mathcal{S})=\underset{S\in \mathcal{S}}{\sum}d^*_H(S).$$
We call the volume of $I(H)$ as the \textit{volume of $H$} and denote it simply by $vol(H)$.
\end{defn}

\begin{defn}\normalfont
	The \textit{modified unified degree matrix} of a hypergraph $H$ with $e$-index $k$, denoted by $\mathbf{U}^{\mathbf{D}}(H)$, is the diagonal matrix of order $k$ whose rows and columns are indexed by the elements $S_1, S_2,\dots,S_k$ of $I(H)$ with $d^*_H(S_i)$ in the $i$-th diagonal entry for all $i=1,2,\dots,k$.
\end{defn}
The modified unified degree matrix of the hypergraph $H$ shown in Figure~\ref{fig} is $diag(4,3,3,3,2,2,$ $3,2,1,2,1,1,1,1)$. 
\begin{figure}[ht]
	\begin{center}
		\includegraphics[scale=1]{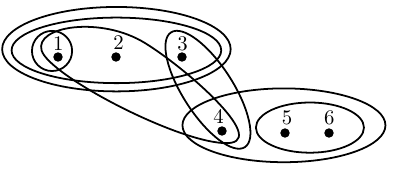}
	\end{center}\caption{The hypergraph $H$}\label{fig}
\end{figure}
\begin{defn}\normalfont
	The \textit{unified Laplacian matrix} of a hypergraph $H$, denoted by $\mathbf{U}^{\mathbf{L}}(H)$, is the matrix $\mathbf{U}^{\mathbf{D}}(H)-\mathbf{U}(H)$. 
\end{defn}
Let $H$ be a hypergraph with $e$-index $k$. Since $\mathbf{U}^{\mathbf{L}}(H)$ is a real symmetric matrix, its eigenvalues are all real. We denote them by $\nu_1(H), \nu_2(H),\dots, \nu_k(H)$ and we shall assume that $\nu_1(H)\geq\nu_2(H)\geq\dots\geq\nu_{k}(H)$.
The characteristic polynomial of $\mathbf{U}^{\mathbf{L}}(H)$ is said to be the \emph{unified Laplacian characteristic polynomial of $H$}. An eigenvalue of $\mathbf{U}^{\mathbf{L}}(H)$ is said to be a \emph{unified Laplacian eigenvalue of $H$} and the spectrum of $\mathbf{U}^{\mathbf{L}}(H)$ is said to be the \textit{unified Laplacian spectrum of $H$}, or simply \textit{$\mathbf{U}^{\mathbf{L}}$-spectrum} of $H$. 

For a loopless hypergraph $H$, the degree of a vertex $S$ of $G_H$ is $d^*_H(S)$ and so $\mathbf{U}^{\mathbf{D}}(H)=D(G_H)$. Consequently $\mathbf{U}^{\mathbf{L}}(H) = L(G_H)$.  Also, notice that if $H$ is a loopless graph, then $\mathbf{U}^{\mathbf{L}}(H)=L(H)$ and so the eigenvalues of $\mathbf{U}^{\mathbf{L}}(H)$ and $L(H)$ are the same. Thereby, we denote the eigenvalues of these two matrices commonly as $\nu_i(H)$, $i=1,2,\dots,k$.
These reveals that the unified Laplacian matrix of a loopless hypergraph is a natural generalization of the Laplacian matrix of a loopless graph.

\begin{defn}\normalfont
	Let $H$ be a loopless hypergraph. An \textit{orientation of an edge} $e$ in $H$ is a $2$-tuple $(S,S')$ such that $\{S,S'\}$ is a $2$-partition of $e$. We say that $S$ and $S'$ are the \textit{tail} and the \textit{head} of $e$, respectively. 
	A hypergraph is said to be \textit{oriented} if each of its edge has an orientation.
	
	Also, we mean an oriented edge $e$ with orientation $(S, S')$ as one side oriented, if $(S',S)$ is not an orientation of $e$. A one side orientation of $e$ is called \textit{an arc of $e$}. 
	Thus, for an edge $e$, there can be at most $\tau(e)$ number of arcs.
\end{defn}

\begin{defn}\normalfont
	Let $H$ be a simple hypergraph with $e$-index $k$. Consider a orientation of each edge of $H$. Let the resulting oriented hypergraph be $\vec{H}$. Let $\theta(H)$ denote the union of the set of all arcs of each edge in $E(H)$. Then the \textit{arc incidence matrix} of $\vec{H}$ (or an \textit{arc incidence matrix} of $H$), denoted by $\mathbf{R}(\vec{H})$ (or $\mathbf{R}(H)$), is the matrix of order $k\times |\theta(H)|$ whose rows and columns are indexed by the elements $S_1,S_2,\dots,S_k$ of $I(H)$ and by the elements $T_1,T_2,\dots,T_{|\theta(H)|}$ of $\theta(H)$, respectively.  For all $i=1,2,\dots,k$; $j=1,2,\dots,|\theta(H)|$,
	
	the $(S_i, T_j)-$th entry of $\mathbf{R}(H)$=$\begin{cases}
	1,& \text{if}~S_i~ \text{is the head in the arc}~T_j;\\
	-1,&	\text{if}~S_i~ \text{is the tail in the arc}~T_j;\\
	0,& \text{otherwise}.
	\end{cases}$
\end{defn}
\begin{note}\normalfont\label{note1}
	Let $H$ be a simple hypergraph.	For each edge in $H$, consider its all possible one side orientations. 
	Since the head and the tail of each orientation of an edge in $H$ are vertices in $G_H$, using these orientations, we can assign orientation to each edge of $G_H$. Then the vertex-arc incidence matrix $R(G_H)$ of $G_H$ is the same as $\mathbf{R}(H)$. Since $L(G_H)=R(G_H)R(G_H)^T$, it follows that $\mathbf{U}^{\mathbf{L}}(H)=\mathbf{R}(H)\mathbf{R}(H)^T$.
\end{note}
Now we start to investigate the properties of the unified Laplacian matrix of a hypergraph.
\begin{lemma}\label{Ob}
	Let $H$ be a hypergraph with $e$-index $k$. Then we have the following.
	\begin{itemize}
		\item[(i)] The row sum and the column sum of $\mathbf{U}^{\mathbf{L}}(H)$ are zero.
		\item [(ii)] If $H$ is simple, then $\mathbf{U}^{\mathbf{L}}(H)$ is positive semi-definite.
		\item[(iii)]If $H$ is simple, then $r(\mathbf{U}^{\mathbf{L}}(H))=r(\mathbf{R}(H))$.
		\item[(iv)] The cofactors of any two elements of $\mathbf{U}^{\mathbf{L}}(H)$ are equal.
		\item[(v)]$\underset{i=1}{\overset{k}{\sum}}\nu_i(H)=vol(H)-\underset{\{v\}\in E^*(H)}{\sum}m(\{v\}) =2\underset{e\in E^*(H)}{\sum}m(e)|\tau(e)|$.
		\item[(vi)] $\underset{i=1}{\overset{k}{\sum}}\nu_i(H)^2=2\left(\underset{e\in E^*(H)}{\sum}m(e)^2|\tau(e)|\right)+\underset{v\in V(H)}{\sum}(d_H(v)-m(\{v\}))^2+\underset{S\in I(H)}{\underset{|S|>1}{\sum}}d^*_H(S)^2$.
	\end{itemize}
\end{lemma}
\begin{proof}
	\begin{itemize}
		\item [(i)] Let $\mathbf{U}(H):=(\mathbf{U}_{SS'})$ and let $S\in I(H)$.
		If $|S|=1$, then by~\cite[Lemma~3.1(i)]{vishnu part1}, the row sum corresponds to $S$ in $\mathbf{U}(H)$ equals $d_H(S)$. If $|S|> 1$, then for each $e\in E^*(H)$ of multiplicity $m$ with $S\subseteq e$ contributes $m$ to the entry $\mathbf{U}_{SS'}$, where $S'=e\backslash\{S\}$. 
		So, the row sum corresponding to $S$ in $\mathbf{U}(H)$ is $d_H(S)$, if $S\subset e$; $d_H(S)-m(S)$, if $S=e$.
		Thus the result follows from the definition of $\mathbf{U}^{\mathbf{L}}(H)$.
		\item[(ii)] Since $H$ is simple, $\mathbf{U}^{\mathbf{L}}(H)=L(G_H)$. So, the result follows from the fact that $L(G_H)$ is positive semi-definite.
		\item[(iii)] Since $H$ is simple, $\mathbf{U}^{\mathbf{L}}(H)=\mathbf{R}(H)\mathbf{R}(H)^T$. So, the result follows from the fact that $r(\mathbf{R}(H)\mathbf{R}(H)^T)=r(\mathbf{R}(H))$.
		\item[(iv)] This follows from part~(i) and \cite[Lemma~4.2]{bapat}: ``Let $X$ be an $n\times n$ matrix with zero row and column sums. Then the
		cofactors of any two elements of $X$ are equal".
		\item[(v)]
		First equality follows from the fact that $\underset{i=1}{\overset{k}{\sum}}\nu_i(H)=tr(\mathbf{U}^{\mathbf{L}}(H))$. Consider~\cite[Equation~(3.6)]{vishnu part1}: 
		\begin{align}\label{dsum2}
		\underset{S\in I(H)}{\sum}d_H(S)=2\left(\underset{e\in E^*(H)}{\sum}m(e)|\tau(e)|\right)+\underset{\{v\}\in E^*(H)}{\sum}m(\{v\})+\partial(H),
		\end{align} where $\partial(H)$ is the cardinality of the multiset of all non-singleton included edges of $H$.
		
		By the definition of $d^*_H(S)$, if $S$ is a non-loop included edge, then $d^*_H(S)$ equals $d_H(S)-m(S)$.	Notice that  
		\begin{align}\label{d*}
		vol(H)=\underset{S\in I(H)}{\sum}d^*_H(S) =\underset{S\in I(H)}{\sum}d_H(S)-\partial(H).
		\end{align} Second equality follows by substituting~\eqref{dsum2} in \eqref{d*}.
		\item[(vi)] Let $\mathbf{U}^{\mathbf{L}}(H)=(l_{S_iS_j})_{1\leq i,j\leq k}$ and let $S_i\in I(H)$ for $i=1,2,\dots,k$. Since $\mathbf{U}^{\mathbf{L}}(H)$ is symmetric, we have
		\begin{equation}\label{EQQ1}
		\mathbf{U}^{\mathbf{L}}(H)^2_{S_iS_i}=\sum_{t=1}^{k}l_{S_iS_t}l_{S_tS_i}=\sum_{t=1}^{k}l_{S_iS_t}^2.
		\end{equation}
		Using~\eqref{EQQ1}, we get
		\begin{align}\label{EQQ2}
		\underset{\nu\in \sigma(\mathbf{U}^{\mathbf{L}}(H))}{\sum}\nu^2=tr(\mathbf{U}^{\mathbf{L}}(H)^2)&=\sum_{i=1}^{k}\mathbf{U}^{\mathbf{L}}(H)^2_{S_iS_i}\nonumber\\
		&=\sum_{i=1}^{k}\sum_{t=1}^{k}l_{S_iS_t}^2\nonumber\\
		&=2\sum_{i<t}l_{S_iS_t}^2+\sum_{i=1}^{k}l_{S_iS_i}^2
		\end{align}
		Notice that $$l_{S_iS_t}=\begin{cases}
		-m(e),& \text{if}~S_i\neq S_t~\text{and}~\{S_i,S_t\}\in\tau(e)~\text{for some}~e\in E^*(H);\\
		d(v)-m(\{v\}),&\text{if}~S_i=S_t=\{v\}~\text{and}~\{v\}\in E^*(H);\\
		d^*_H(S),&\text{if}~S_i=S_t~\text{and}~|S_i|>1;\\
		0,&\text{otherwise}.
		\end{cases}$$	
		The result follows by substituting $l_{S_iS_t}$ in~\eqref{EQQ2}.
	\end{itemize}
\end{proof}
\begin{thm}
	Let $H$ be a loopless $m$-uniform hypergraph with $e$-index $k$. Then,
	
	$|E(H)|=\frac{1}{m}\left(\underset{i=1}{\overset{k}{\sum}}\nu_i(H)-\underset{\underset{|S|\geq2}{S\in I(H),}}{\sum}d^*_H(S)\right)$.
\end{thm}
\begin{proof}
	Since $H$ has no loops, $m>1$. By Lemma~\ref{Ob}$(v)$, we have
	\begin{align}
	\underset{i=1}{\overset{k}{\sum}}\nu_i(H)&=\underset{S\in I(H)}{\sum}d^*_H(S)\nonumber\\
	&=\underset{\underset{|S|\geq2}{S\in I(H),}}{\sum}d^*_H(S)+\underset{\underset{|S|=1}{S\in I(H),}}{\sum}d^*_H(S)\nonumber\\
	&=\underset{\underset{|S|\geq2}{S\in I(H),}}{\sum}d^*_H(S)+\underset{v\in V(H)}{\sum}d_H(v).\label{eq2}
	\end{align}
	According to~\cite[Lemma~3.1(ii)]{vishnu part1}: If $H$ is $m$-uniform, then 
	\begin{align}\label{L eq1}
\underset{v\in V(H)}{\sum}d_H(v)=m\cdot|E(H)|.
	\end{align}
	Applying~\eqref{L eq1} in~\eqref{eq2}, we get the result.
\end{proof}


\subsection{Bounds on the unified Laplacian eigenvalues of a hypergraph}
In this subsection, we bound the largest, second smallest unified Laplacian eigenvalue of a hypergraph by using its invariants. Also, we introduce the exact set diameter of a hypergraph and constrain it by using its unified Laplacian eigenvalues.
\begin{defn}\normalfont
	Let $H$ be a hypergraph and let $S\in I(H)$. \textit{The $e$-neighborhood of $S$} in $H$, denoted by $N_H(S)$, is the set of all neighbors of $S$ in $H$, i.e., $N_H(S)=\{S'\in I(H)~|~S\sim S'\}$.
\end{defn}
Further, we denote 
\begin{itemize}
	\item $m^*(H):=\underset{S\in I(H)}{\min}d^*_H(S)$,
	
	\item  $\hat{d^*}(H):=\underset{S\in I(H)}{\max}d^*_H(S)$, 
	
	\item $\overline{d^*}(H):=\displaystyle\frac{1}{k}\underset{S\in I(H)}{\sum}d^*_H(S)$ and  we call $\overline{d^*}(H)$ as the \textit{average modified unified degree} of $H$. 
\end{itemize}
 
\begin{thm}
	Let $H$ be a simple hypergraph with $e$-index $k~(\geq 2)$.
	Then,
	\[\nu_{k-1}(H)\leq\displaystyle\left(\frac{k}{k-1}\right)m^*(H) \leq\frac{2}{k-1}\underset{e\in E(H)}{\sum}|\tau(e)|.\]
\end{thm}
\begin{proof}
	By Lemma~\ref{Ob}$(i)$, it is clear that $L(H)J_{ k\times1}=0$.  Therefore, the first inequality follows from Lemma~\ref{Ob}$(ii)$ and \cite[Lemma]{Fiedler}:``Let $M=(m_{ij})$ be a symmetric positive semi-definite $n\times n$ matrix and $x\in \mathbb{R}^n$ be such
	that $Mx=\mathbf{0}$. 
	Then the second smallest eigenvalue $\lambda$ of $M$ satisfies $\lambda\leq (\frac{n}{n-1})\underset{i}{\min}\{m_{ii}\}$".
	
It is clear that
	\begin{equation}\label{eq1}
	k\min\{d^*_H(S_i)~|~S_i\in I(H)~\text{for}~i=1,2,\dots,k\}\leq \underset{i=1}{\overset{k}{\sum}}d^*_H(S_i)=vol(H).
	\end{equation}
	Dividing by $k-1$ on both sides of~\eqref{eq1}, and applying Lemma~\ref{Ob}$(v)$ in it, we get the second inequality. 	
\end{proof}
\begin{thm}
	Let $H$ be a simple hypergraph with $e$-index $k$. Let $I(H)=\{S_1,S_2,\dots,S_k\}$.
	Then, we have the following.
	\begin{itemize}
		\item [(i)] If $d^*_H(S_1)\geq d^*_H(S_2)\geq\dots\geq d^*_H(S_k)$, then $\underset{i=1}{\overset{t}{\sum}}\nu_i(H)\geq\underset{i=1}{\overset{t}{\sum}}d^*_H(S_i)$ for $t\in\{1,2,\dots,k\}$. The equality holds when $t=k$.		
		\item[(ii)] $\big(\frac{k-1}{k}\big)\nu_{k-1}(H)\leq \overline{d^*}(H)\leq\big(\frac{k-1}{k}\big)\nu_1(H)$.
	\end{itemize} 
\end{thm}
\begin{proof}
	For each $S_i\in I(H)$, $d^*_H(S_i)$ is the same as the degree of the vertex $S_i$ in $G_H$. Moreover, the mean degree of $G_H$ is the same as the average modified unified degree of $H$.
	Thus the results~(i) and (ii) directly follows from~\cite[p. 186, Theorem~7.1.3, Remark~7.1.4]{cvetko}.
\end{proof}
In the following result, we asserts some bounds on the largest unified Laplacian eigenvalue of a simple hypergraph.
\begin{pro}	Let $H$ be a simple hypergraph with $e$-index $k$. Then, we have the following.
	\begin{itemize}
		\item [(i)] $\nu_1(H)\leq k$.
		\item[(ii)]$\nu_1(H)\leq \max\{d^*_H(S)+d^*_H(S')~|~S,S'\in I(H)~\text{and}~S\sim S'\}$.
		\item[(iii)] $\displaystyle\nu_1(H)\leq \max\biggl\{\frac{d^*_H(S)[d^*_H(S)+\zeta(S)]+d^*_H(S')[d^*_H(S')+\zeta(S')]}{d^*_H(S)+d^*_H(S')}~\bigg|~S,S'\in I(H)~\text{and}~S\sim S'\biggr\}$, where $\zeta(S)=\frac{1}{|N_H(S)|}\underset{S'\in N_H(S)}{\sum}d^*_H(S')$.
		\item[(iv)] $\nu_1(H)\geq \max\{\sqrt{(d^*_H(S)-d^*_H(S'))^2+4U_{SS'}}~|~S,S'\in I(H)~\text{and}~S\neq S'\}$, where $\mathbf{U}_{SS'}$ is the $(S,S')$-th entry of $\mathbf{U}(H)$.
	\end{itemize}
\end{pro}
\begin{proof}
	Since the $\mathbf{U}^{\mathbf{L}}$-spectrum of the hypergraph $H$ and the spectrum of $L(G_H)$ are the same, the proofs directly follows from~\cite[Proposition~7.3.3 and Theorems~7.3.4, 7.3.5, 7.3.6]{cvetko}. 
\end{proof}
\begin{lemma}\label{Knarg}
	Let $H$ be a hypergraph with $e$-index $k$. Then we have the following.
	\begin{itemize}
		\item[(i)] $G_H\cong P_k$ if and only if $H\cong P_k$.
		\item[(ii)] $G_H\cong C_k$ if and only if $H\cong C_k$.	
		\item[(iii)] $G_H\cong K_k$ if and only if $H\cong K_k$.
		\item[(iv)]	$G_H\cong K_{n_1,n_2,\dots,n_r}$ if and only if $H\cong K_{n_1,n_2,\dots,n_r}$.
	\end{itemize}
\end{lemma}
\begin{proof}
	It is clear that if $H$ is a graph, then $G_H\cong H$. So, in all the parts, it is enough to prove the other direction.
	\begin{itemize}
		\item[(i)] We assume that $G_H$ is a path. Suppose $H$ contains an edge of cordinality greater than $2$, then $G_H$ contains atleast three vertices of degree one, which is not possible, since $G_H$ is a path. Thus, $H$ is a graph and so $H\cong G_H$.
		\item [(ii)]   We assume that $G_H$ is a cycle. Suppose $H$ contains an edge of cordinality greater than $2$, then $G_H$ must contain an edge that does not belong to $G_H$, which is not possible, since $G_H$ is a cycle. Thus, $H$ is a graph and so $H\cong G_H$.
		
		\item[(iii)] We assume that $G_H\cong K_k$. Then clearly any two vertices of $H$ are joined by an edge.  
		Suppose $H$ has an edge $e$ of cordinality greater than $2$. Then there exists a pair of non-disjoint proper subsets of $e$ such that the vertices in $G_H$ corresponding to this pair are not adjacent, a contradiction. Thus, $H$ has no edge of cordinality greater than $2$. So, $H$ is the complete graph on $k$ vertices.
		
		\item[(iv)] We assume that $G_H$ is a complete multipartite graph with vertex set partitions $X_i$, $i=1,2,\dots,r$. First we show that $H$ is a graph. Suppose $H$ has an edge $e$ of cordinality greater than two. Let $S_1$ and $S_2$ be parts of $e$. Then $S_1, S_2\notin X_i$ for any $i=1,2,\dots,r$, since $G_H$ is a multipartite graph. Without loss of generality, we assume that $S_1\in X_1$ and $S_2\in X_2$. Since $|e|>2$, either $S_1$ or $S_2$ is of cordinality greater than one and so there exists $S_3\in I(H)$ such that $S_3\cap S_i\neq \Phi$ for each $i=1,2$. Thus, being $G_H$ is a complete multipartite graph, $S_3\notin X_i$ for any $i=1,2,\dots,r$, which is not possible. Hence, $H$ cannot have any edge of cordinality greater than two. Therefore, $H$ is a graph. Proof is complete, since for a graph $H$, its associated graph $G_H$ is isomorphic to $H$. 
	\end{itemize}
\end{proof}

Let $H$ be a hypergraph. Let $e\in E(H)$. For, any integer $r$, $0<r\leq m(e)$, $H-e^r$ denotes the hypergraph obtained from $H$ by deleting the $r$ copies of the edge $e$ from $H$.
\begin{thm}\label{L thm}
	Let $H$ be a hypergraph with $e$-index $k$. Suppose $H$ has a non-loop edge $e$ such that $I(H-e^r)=I(H)$, where $0<r\leq m(e)$. Then the following hold:\\
	If $|e|=2$, then
	\begin{center} $0=\nu_k(H-e^r)=\nu_k(H)\leq \nu_{k-1}(H-e^r)\leq \nu_{k-1}(H)\leq\dots\leq\nu_2(H)\leq\nu_{1}(H-e^r)\leq \nu_{1}(H)$;
	\end{center} otherwise,
	\begin{center} $0=\nu_k(H-e^r)=\nu_k(H)\leq \nu_{k-1}(H-e^r)+2r\leq \nu_{k-1}(H)+2r\leq\nu_{k-2}(H-e^r)+4r\leq \nu_{k-2}(H)+4r\leq\dots\leq\nu_2(H)+2r(k-2)\leq\nu_{1}(H-e^r)+2r(k-1)\leq \nu_{1}(H)+2r(k-1)$.
	\end{center}
\end{thm}
\begin{proof}
	If needed, we rearrange the elements of $I(H)$ as $S_1,S_2,\dots,S_t,S_{t+1},\dots,S_k$ such that $\{S_i,S_{i+1}\}\in\tau(e)$, for $i=2,3,\dots,t-1$, where $t=2|\tau(e)|$. 
	Since $I(H-e^r)=I(H)$, 	we can see that 
	\begin{align}\label{eq*}
	\mathbf{U}^{\mathbf{L}}(H-e^r)=\mathbf{U}^{\mathbf{L}}(H)-N,
	\end{align} where \begin{align} 
	N=\begin{pmatrix}  
	\begin{array}{c|c}
	\begin{pmatrix}
	r & -r & 0 & 0& \cdots& 0 &0\\
	-r & r &0 & 0& \cdots & 0 & 0\\
	0 & 0 & r & -r& \cdots& 0 &0\\
	0 & 0 &-r & r& \cdots & 0 & 0\\
	\vdots&\vdots&\vdots&\vdots&\ddots&\vdots&\vdots\\
	0 & 0 & 0 & 0& \cdots& r &-r\\
	0 & 0 &0 & 0& \cdots & -r & r\\
	\end{pmatrix}_{t \times t} &\textbf{0}\\
	\hline 
	\textbf{0} &\textbf{0}  
	\end{array}  
	\end{pmatrix}_{k\times k}.\end{align}
	
	The eigenvalues of $N$ are $0$ with multiplicity $k-\frac{t}{2}$ and $2r$ with multiplicity $\frac{t}{2}$. Notice that the second largest eigenvalue $\lambda_2(N)$ of $N$ is non-zero if and only if $t=2$. 
	Thus, we have
	\begin{align}\label{lambda2}
	\lambda_2(N)=\begin{cases}
	0, ~\text{if}~|e|=2;\\	
	2r,~\text{otherwise}.
	\end{cases}\end{align}
	
	It is known~(c.f. \cite[Theorem~1.3.15]{cvetko}) that if $A$ and $B$ are $n\times n$ Hermitian matrices, then 
	\begin{eqnarray}
	\lambda_i(A+B)\leq \lambda_j(A)+\lambda_{i-j+1}(B) ~\text{for}~ n\geq i\geq j\geq 1;\label{e1}\\
	\lambda_i(A+B)\geq \lambda_j(A)+\lambda_{i-j+n}(B) ~\text{for}~ n\geq j\geq i\geq 1.\label{e2}
	\end{eqnarray}
	Now, taking $A=\mathbf{U}^{\mathbf{L}}(H-e^r)$ and $B=N$, we get from~\eqref{eq*} that $A+B=\mathbf{U}^{\mathbf{L}}(H)$.
	Applying these in inequalities~\eqref{e1} and~\eqref{e2}, we get
	\begin{eqnarray}
	\nu_i(H)\leq \nu_j(H-e^r)+\lambda_{i-j+1}(N) ~\text{for}~ k\geq i\geq j\geq 1;\label{e3}\\
	\nu_i(H)\geq \nu_j(H-e^r)+\lambda_{i-j+k}(N) ~\text{for}~ k\geq j\geq i\geq 1.\label{e4}
	\end{eqnarray}
	Taking $i=j$ in \eqref{e4} and since $\lambda_k(N)=0$, we get
	\begin{align}
	\nu_i(H)\geq \nu_i(H-e^r)~\text{for}~ i=1,2,\dots,k.\label{e5}
	\end{align}
	Taking $i=j+1$ in \eqref{e3} and applying the value of $\lambda_2(N)$ given in~\eqref{lambda2} in it for $j=1,2,\dots,k-1$, we get
	\begin{eqnarray}
	\nu_{j+1}(H)\leq \nu_j(H-e^r),~\text{if}~|e|=2;\label{e6}\\
	\nu_{j+1}(H)\leq \nu_j(H-e^r)+2r,~\text{otherwise.}\label{e7}	 
	\end{eqnarray}
	It is clear that  $\nu_{k}(H)=\nu_k(H-e^r)=0$.
	
	If $|e|=2$, then the result follows from~\eqref{e5} and~\eqref{e6}.
	Otherwise, we have the following inequalities by substituting $j=1,2,\dots,k-1$ in~\eqref{e7} successively and applying~\eqref{e5}:
	\begin{align*}
	\nu_{t}(H)\leq \nu_{t-1}(H-e^r)+2r\leq\nu_{t-1}(H)+2r~ \text{for}~t=2,3,\dots,k.
	\end{align*}
	Result follows from the above inequalities.
\end{proof}

\subsubsection{New structures and invariants of hypergraphs and unified Laplacian eigenvalues}
\begin{defn}\normalfont
We call an exact path in which all the edges are distinct as an \textit{edge exact path}.
	An exact path of length at least two in a hypergraph is said to be an \textit{internal unified path} if all pairs of its parts are disjoint except the pair that contains its initial part and terminal part. 
\end{defn}
\begin{defn}\normalfont\label{eec}
	A hypergraph $H$ is said to be \textit{edge exact connected} if for any two distinct vertices $u$ and $v$ in $H$, there exists $S, S'\in I(H)$ with $u\in S$ and $v\in S'$ such that $S$ and $S'$ are joined by an edge exact path in $H$. 
\end{defn}
\begin{defn}\normalfont\label{eed}
	Let $H$ be a hypergraph. The \textit{edge exact distance} or simply the \textit{$ee$-distance} between distinct vertices $u$ and $v$ in $H$, denoted by $eed_H(u,v)$, is the length of a shortest edge exact path joining $S, S'\in I(H)$ with $u\in S$ and $v\in S'$, if such an edge exact path exists; $\infty$, otherwise. For a vertex $u$ in $H$, we define $eed_H(u,u)=0.$ 
\end{defn}

\begin{defn}\normalfont\label{eeD}
	The \textit{edge exact diameter} of an edge exact connected hypergraph $H$ is the maximum of the $ee$-distance between all the pairs of vertices in $H$ and we denote it by $EED(H)$,   
	i.e., $EED(H)=\max\{eed_H(u,v)~|~u,v\in V(H)\}$.
\end{defn}

The hypergraph shown in Figure~\ref{fig2} is edge exact connected. It is clear that the edge exact distance defined on an edge exact connected hypergraph $H$ is a semi-metric. However, it is not a metric, as the triangle inequality may fail. For instance, in the hypergraph $H$ shown in Figure~\ref{fig2}, $eed_H(1,6)=3$ and $eed_H(6,13)=1$; but $eed_H(1,13)=5$. 

 \begin{figure}[ht]
	\begin{center}
		\includegraphics[scale=1.15]{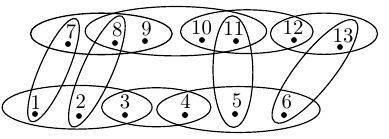}
	\end{center}\caption{The hypergraph $H$}\label{fig2}
\end{figure}


\begin{defn}\normalfont\label{iuc}
	A hypergraph $H$ is said to be \textit{inter-uni-connected} if for any two distinct vertices $u$ and $v$ in $H$, there exists $S, S'\in I(H)$ with $u\in S$ and $v\in S'$ such that $S$ and $S'$ are joined by an internal unified path in $H$. 
\end{defn}
\begin{defn}\normalfont\label{iud}
	Let $H$ be a hypergraph. The \textit{internal unified distance} or simply the \textit{$iu$-distance} between distinct vertices $u$ and $v$ in $H$, denoted by $iud_H(u,v)$, is the length of a shortest internal unified path joining $S, S'\in I(H)$ with $u\in S$ and $v\in S'$, if such an internal unified path exists; $\infty$, otherwise. For a vertex $u$ in $H$, we define $iud_H(u,u)=0.$ 
\end{defn}

\begin{defn}\normalfont\label{iuD}
	The \textit{internal unified diameter} of an inter-uni-connected hypergraph $H$ is the maximum of the $iu$-distance between all the pairs of vertices in $H$, and we denote it by $IUD(H)$,   
	i.e., $IUD(H)=\max\{iud_H(u,v)~|~u,v\in V(H)\}$.
\end{defn}
The hypergraph $H$ with $V(H)=\{1,2,3,4,5,6,7,8,9,10,11,12,13,14\}$ and $E(H)=\{\{1,7\},$ $\{1,8\},\{2,8\},\{3,4\},\{5,11\},\{5,12\},\{5,13\},\{5,14\},\{6,13\},\{6,14\},\{9,12\},\{12,13\},\{13,14\},$ 

\noindent$\{1,2,3\},\{4,5,6\},\{7,8,9\}, \{8,9,10,11\},\{10,11,12\}\}$ is inter-uni-connected. 
It is clear that the internal unified distance defined on an inter-uni-connected hypergraph $H$ is a semi-metric. However, it is not a metric, as the triangle inequality may fail.
For instance, in the inter-uni-connected hypergraph $H$ mentioned above, we have 
$iud_H(1,6)=3$ and $iud_H(6,14)=2$; but $iud_H(1,14)=6$. 


\begin{defn}\normalfont\label{uc}
	A hypergraph $H$ is said to be \textit{uni-connected} if for any two distinct vertices $u$ and $v$ in $H$, there exists $S, S'\in I(H)$ with $u\in S$ and $v\in S'$ such that $S$ and $S'$ are joined by an unified path in $H$.
\end{defn}
\begin{defn}\normalfont\label{ud}
	Let $H$ be a hypergraph. The \textit{unified distance} or simply the \textit{$u$-distance} between distinct vertices $u$ and $v$ in $H$, denoted by $ud_H(u,v)$, is the length of a shortest unified path joining $S, S'\in I(H)$ with $u\in S$ and $v\in S'$, if such an unified path exists; $\infty$, otherwise. For a vertex $u$ in $H$, we define $ud_H(u,u)=0.$ 
\end{defn}

\begin{defn}\normalfont\label{uD}
	The \textit{unified diameter} of an uni-connected hypergraph $H$ is the maximum of the $u$-distance between all the pairs of vertices in $H$, and we denote it by $UD(H)$,   
	i.e., $UD(H)=\max\{ud_H(u,v)~|~u,v\in V(H)\}$.
\end{defn}
The hypergraph $H$ with $V(H)=\{1,2,3,4,5,6,7,8,9,10\}$ and $E(H)=\{\{1,2\},\{1,3\},\{1,7\},$ $\{2,3\},\{2,7\},\{3,4\},\{4,9\},\{5,6\}, \{5,9\},\{5,10\},\{6,10\},\{7,8\},\{8,9\},\{1,2,3\},\{2,3,4\},\{3,4,5\},$ 

\noindent$\{3,5,6\}$ is uni-connected. It is clear that the unified distance defined on a uni-connected hypergraph $H$ is a semi-metric. However it is not a metric, as the triangle inequality may fail.
For instance, in the uni-connected hypergraph $H$ mentioned above, $ud_H(1,4)=2$ and $ud_H(4,6)=2$; but $ud_H(1,6)=5$. 


\begin{defn}\normalfont\label{sec}
	A hypergraph $H$ is said to be \textit{strong exactly connected} if for any two distinct vertices $u$ and $v$ in $H$, there exists an exact path joining $u$ and $v$ in $H$.
\end{defn}
\begin{defn}\normalfont\label{sed}
	Let $H$ be a hypergraph. The \textit{strong exact distance} or simply the \textit{$se$-distance} between distinct vertices $u$ and $v$ in $H$, denoted by $sed_H(u,v)$, is the length of a shortest exact path joining $u$ and $v$, if such an exact path exists; $\infty$, otherwise. For a vertex $u$ in $H$, we define $sed_H(u,u)=0.$ 
\end{defn}

\begin{defn}\normalfont\label{seD}
	The \textit{strong exact diameter} of an strong exactly connected hypergraph $H$ is the maximum of the $se$-distance between all the pairs of vertices in $H$, and we denote it by $SED(H)$,   
	i.e., $SED(H)=\max\{sed_H(u,v)~|~u,v\in V(H)\}$.
\end{defn}
A strong exactly connected hypergraph is shown in Figure~\ref{fig5}.
The strong exact distance defined on a strong exactly connected hypergraph is a metric: For triangle inequality, let $u,v\in V(H)$. Suppose there exist $w\in V(H)$ such that $sed_H(u,v)> sed_H(u,w)+sed_H(w,v)$. This assures that there exists an exact path joining $u$ to $v$ via $w$ of length less than $sed_H(u,v)$, which is not possible. The remaining two conditions of a metric are straight forward.

\begin{defn}\normalfont\label{seec}
	A hypergraph $H$ is said to be \textit{strong edge exact connected} if for any two distinct vertices $u$ and $v$ in $H$, there exists an edge exact path joining $u$ and $v$ in $H$.
\end{defn}
\begin{defn}\normalfont\label{seed}
	Let $H$ be a hypergraph. The \textit{strong edge exact distance} or simply the \textit{$see$-distance} between distinct vertices $u$ and $v$ in $H$, denoted by $seed_H(u,v)$, is the length of a shortest edge exact path joining $u$ and $v$, if such an edge exact path exists; $\infty$, otherwise. For a vertex $u$ in $H$, we define $seed_H(u,u)=0.$ 
\end{defn}

\begin{defn}\normalfont\label{seeD}
	The \textit{strong edge exact diameter} of an strong edge exact connected hypergraph $H$ is the maximum of the $see$-distance between all the pairs of vertices in $H$, and we denote it by $SEED(H)$,   
	i.e., $SEED(H)=\max\{seed_H(u,v)~|~u,v\in V(H)\}$.
\end{defn}

The hypergraph $H$ with $V(H)=\{1,2,3,4,5,6\}$ and $E(H)=\{\{1,2\}, \{2,3\},\{3,4\},\{4,5\},$ $\{1,2,3\},\{1,2,5\},\{1,3,5\}, \{2,3,4\},\{2,4,6\},\{3,4,6\}\}$ is strong edge exact connected. 
It is clear that the strong edge exact distance defined on a strong edge exact connected hypergraph $H$ is a semi-metric. However, we are neither able to prove the triangle inequality nor to find a strong edge exact connected hypergraph in which the strong edge exact distance defined on it violates the triangle inequality.


\begin{defn}\normalfont\label{suc}
	A hypergraph $H$ is said to be \textit{strong uni-connected} if for any two distinct vertices $u$ and $v$ in $H$, there exists an unified path joining $u$ and $v$ in $H$.
\end{defn}
\begin{defn}\normalfont\label{sud}
	Let $H$ be a hypergraph. The \textit{strong unified distance} or simply the \textit{$su$-distance} between distinct vertices $u$ and $v$ in $H$, denoted by $sud_H(u,v)$, is the length of a shortest unified path joining $u$ and $v$, if such an unified path exists; $\infty$, otherwise. For a vertex $u$ in $H$, we define $sud_H(u,u)=0.$ 
\end{defn}

\begin{defn}\normalfont\label{suD}
	The \textit{strong unified diameter} of a strong uni-connected hypergraph $H$ is the maximum of the $su$-distance between all the pairs of vertices in $H$, and we denote it by $SUD(H)$,   
	i.e., $SUD(H)=\max\{sud_H(u,v)~|~u,v\in V(H)\}$.
\end{defn}

The hypergraph $H$ with $V(H)=\{1,2,3,4,5,6,7,8,9,10\}$ and $E(H)=\{\{1,2\},\{1,3\},\{1,7\},$ $\{2,3\},\{2,7\},\{3,4\},\{4,9\},\{5,6\}, \{5,9\},\{5,10\},\{6,10\},\{7,8\},\{8,9\},\{1,2,3\},\{2,3,4\},\{3,4,5\},$ 

\noindent$\{3,5,6\}$ is strong uni-connected. It is clear that the strong unified distance defined on a strong uni-connected hypergraph $H$ is a semi-metric. However, it is not a metric, as the triangle inequality may fail.
For instance, in the strong uni-connected hypergraph $H$ mentioned above,  $sud_H(1,4)=2$ and $sud_H(4,6)=2$; but $sud_H(1,6)=5$.


\begin{defn}\normalfont\label{dpc}
	A hypergraph $H$ is said to be \textit{deeply connected} if every pair of distinct elements of $I(H)$ are joined by an exact path.
\end{defn}

\begin{defn}\normalfont\label{esd}
	Let $H$ be a hypergraph. The \textit{exact set distance} between the distinct elements $S$ and $S'$ of $I(H)$, denoted by $esd_H(S,S')$, is the length of a shortest exact path joining them in $H$ if such an exact path exist; $\infty$,  otherwise. For $S\in I(H)$, we define $esd_H(S,S)=0.$ 
\end{defn}
\begin{defn}\normalfont\label{esD}
	The \textit{exact set diameter} of a deeply connected hypergraph $H$ is the maximum of the exact set distance between all the pairs of elements in $I(H)$, and we denote it by $ESD(H)$, i.e., $ESD(H)=\max\{esd_H(S,S')~|~S,S'\in I(H)\}.$
\end{defn}

A deeply connected hypergraph is shown in Figure~\ref{fig5}. The exact set distance between $S$ and $S'$ in $I(H)$ is the same as the distance between the vertices $S$ and $S'$ in the associated graph $G_H$. Since the distance between vertices defined in a connected graph is a metric, it follows that the exact set distance defined on a deeply connected hypergraph $H$ is a metric.

It is clear that a hypergraph $H$ is deeply connected if and only if its associated graph $G_H$ is connected.

\begin{defn}\normalfont\label{dpeec}
	A hypergraph $H$ is said to be \textit{deeply edge exact connected} if every pair of distinct elements of $I(H)$ are joined by an edge exact path.
\end{defn}
\begin{defn}\normalfont\label{eesd}
	Let $H$ be a hypergraph. The \textit{edge exact set distance} between the distinct elements $S$ and $S'$ of $I(H)$, denoted by $eesd_H(S,S')$, is the length of a shortest edge exact path joining them in $H$ if such an edge exact path exist; $\infty$,  otherwise. For $S\in I(H)$, we define $eesd_H(S,S)=0.$ 
\end{defn}
\begin{defn}\normalfont\label{eesD}
	The \textit{edge exact set diameter} of a deeply edge exact connected hypergraph $H$ is the maximum of the edge exact set distance between all the pairs of elements in $I(H)$, and we denote it by $EESD(H)$, i.e., $EESD(H)=\max\{eesd_H(S,S')~|~S,S'\in I(H)\}.$
\end{defn}

The hypergraph $H$ with $V(H)=\{1,2,3,4\}$ and $E(H)=\{\{1,2\}, \{1,3\},\{1,4\},\{1,2,3\}, \{1,2,4\},$ $\{2,3,4\}\}$ is deeply edge exact connected. It is clear that the edge exact set distance defined on a deeply edge exact connected hypergraph $H$ is a semi-metric. However, it is not a metric, since the triangle inequality may fail.
For instance, in the deeply edge exact connected hypergraph $H$ mentioned above, $eesd_H(\{1,4\},\{2\})=1$ and $eesd_H(\{2\},\{2,4\})=2$; but $eesd_H(\{1,4\},\{2,4\})=4$. 

\begin{defn}\normalfont\label{dpiuc}
	A hypergraph $H$ is said to be \textit{deeply inter-uni-connected} if every pair of distinct elements of $I(H)$ are joined by an internal unified path.
\end{defn}
\begin{defn}\normalfont\label{iusd}
	Let $H$ be a hypergraph. The \textit{internal unified set distance} between the distinct elements $S$ and $S'$ of $I(H)$, denoted by $iusd_H(S,S')$, is the length of a shortest internal unified path joining them in $H$ if such an internal unified path exist; $\infty$,  otherwise. For $S\in I(H)$, we define $iusd_H(S,S)=0.$ 
\end{defn}
\begin{defn}\normalfont\label{iusD}
	The \textit{internal unified set diameter} of a deeply inter-uni connected hypergraph $H$ is the maximum of the internal unified set distance between all the pairs of elements in $I(H)$, and we denote it by $IUSD(H)$, i.e., $IUSD(H)=\max\{iusd_H(S,S')~|~S,S'\in I(H)\}.$
\end{defn}

The hypergraph $H$ with $V(H)=\{1,2,3,4,5\}$ and $E(H)=\{\{1,2\}, \{1,3\},\{1,4\},\{1,5\},\{2,3\},$ $\{2,4\},\{2,5\},\{3,4\},\{3,5\},\{4,5\},\{1,2,3\},\{1,2,4\},\{1,2,5\},\{1,3,4\},\{1,3,5\},\{1,4,5\},\{2,3,4\},$

\noindent $\{2,3,5\},\{2,4,5\},\{3,4,5\}\}\}$ is deeply inter-uni connected. It is clear that the internal unified set distance defined on a deeply inter-uni connected hypergraph $H$ is a semi-metric.  However, we are neither able to prove the triangle inequality nor to find a deeply inter-uni connected hypergraph in which the internal unified set distance defined on it violates the triangle inequality.


The relationships among the different types of connectedness of a hypergraph defined above is given in Figure~\ref{diag}.
\begin{figure}[ht]
	\begin{center}
			\includegraphics[scale=0.9]{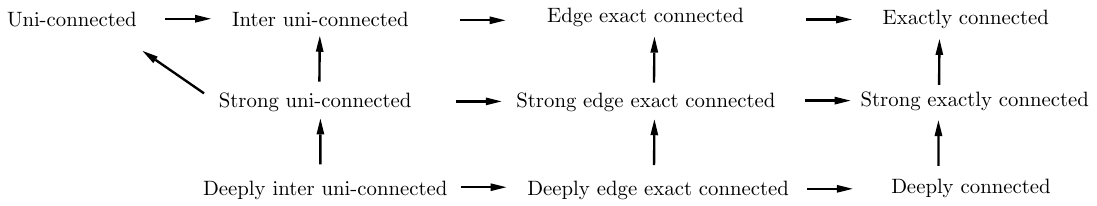}
		\end{center}\caption{The relationships among different types of connectedness of a hypergraph}\label{diag}
\end{figure}

In a graph $G$, Definitions~\ref{eec}, \ref{iuc}, \ref{uc}, \ref{sec}, \ref{seec}, \ref{suc}, \ref{dpc}, \ref{dpeec} and \ref{dpiuc} coincide with the definition of connectedness; Definitions~\ref{eed}, \ref{iud}, \ref{ud}, \ref{sed}, \ref{seed}, \ref{sud}, \ref{esd}, \ref{eesd}, \ref{iusd} coincide with the definition of distance between two vertices; and Definitions~\ref{eeD}, \ref{iuD}, \ref{uD}, \ref{seD}, \ref{seeD}, \ref{suD}, \ref{esD}, \ref{eesD}, \ref{iusD} coincide with the definition of diameter.

\begin{thm}\normalfont \label{chain1}
	Let $H$ be a hypergraph. Then we have the following.
	\begin{itemize}
		\item[(i)] If $H$ is uni-connected, then $UD(H)\geq IUD(H)\geq EED(H)\geq ED(H)$;
		\item[(ii)] If $H$ is strong uni-connected, then $SUD(H)\geq SEED(H)\geq SED(H)$;
		\item[(iii)]If $H$ is deeply inter-uni-connected, then $IUSD(H)\geq EESD(H)\geq ESD(H)$.
	\end{itemize}
\end{thm}
\begin{proof} 
	 We prove the first inequality of part~(i). The remaining two inequalities of part~(i) can be established similarly. 
		
		If $H$ is uni-connected, then it is also inter-uni connected, edge exact connected, and exactly connected.
		 Additionally, as an unified path is an inter-unified path, we have $ud_H(u,v)\geq iud_H(u,v)$ for all $u, v\in V(H)$. It follows that $UD(H)\geq IUD(H)$. 
		 
	 The remaining two parts can be proved in a similar way as part~(i) using the relationship shown in Figure~\ref{diag}. 		  		
\end{proof}
\begin{thm}\normalfont\label{chain2}
	Let $H$ be a hypergraph. Then we have the following.
	\begin{itemize}
		\item[(i)] If $H$ is deeply inter-uni-connected, then $IUSD(H)\geq SUD(H)\geq UD(H)$;
		\item[(ii)] If $H$ is deeply edge exact connected, then $EESD(H)\geq SEED(H)\geq EED(H)$;
		\item[(iii)]If $H$ is deeply connected, then $ESD(H)\geq SED(H)\geq ED(H)$.
	\end{itemize}
\end{thm}
\begin{proof} 
 If $H$ is deeply inter-uni-connected, then it is strong uni-connected and  uni-connected. 
	Notice that $\{sud_H(u,v)\mid u,v\in V(H)\}\subseteq \{iusd_H(S,S')\mid S,S'\in I(H)\}$. It follows that $IUSD(H)\geq SUD(H)$. 
	Also, we have $sud_H(u,v)\geq ud_H(u,v)$ for all $u, v\in V(H)$. It follows that $SUD(H)\geq UD(H)$. Thus proof of part~(i) is completed.
	
The remaining two parts can be proved in a similar way as part~(i) using the relationship shown in Figure~\ref{diag}. 	
\end{proof}

\begin{thm}\label{thm 5.4}
	Let $H$ be a simple hypergraph with $e$-index $k$.
	\begin{itemize}
		\item [(i)]	 If $S\nsim S'$ in $I(H)$, then $\nu_{k-1}(H)\leq\frac{1}{2}(d^*_H(S)+d^*_H(S'))$. In particular, if $H\neq K_k$, then $\nu_{k-1}(H)\leq k-2$.
		\item[(ii)] If $H$ is deeply connected with exact diameter $d$, then \[\nu_{k-1}(H)\leq \hat{d^*}(H)-2\sqrt{\hat{d^*}(H)-1}+\frac{2}{\lfloor \frac{d}{2}\rfloor}\left(\sqrt{\hat{d^*}(H)-1}-1\right).\]
	\end{itemize}
\end{thm}
\begin{proof}
	\begin{itemize}
		\item[(i)]
		If $S\nsim S'$ in $I(H)$, then $S$ and $S'$ are not adjacent in $G_H$. 
		Since the spectrum of $\mathbf{U}^{\mathbf{L}}(H)$ and the spectrum of $L(G_H)$ are the same, the proof follows from part~$(iii)$ of Lemma~\ref{Knarg} and from~\cite[Theorem 7.4.4]{cvetko}: ``If $u$ and $v$ are two non-adjacent vertices of a simple graph $G$ on $n$ vertices, then
		$\nu_{n-1}(G)\leq\frac{1}{2}(d_G(u) + d_G(v) )$. In particular, if $G$ is not complete, then $\nu_{n-1}(G)\leq n-2$".
		\item[(ii)] Since $H$ is deeply connected, $G_H$ is connected.
		Notice that the exact distance between two vertices $u,v$ in $H$ is the same as the distance between the vertices $S_i, S_j$ in $G_H$, where $S_i,S_j\in I(H)$ with $u\in S_i, v\in S_j$. Also, by the definition of exact distance and exact set distance of a hypergraph, we have \begin{align*}
		\{ed_H(u,v)~|~u,v\in V(H)\}&\subseteq\{esd_H(S,S')~|~S,S'\in I(H)\}\\
		&=\{d_{G_H}(S,S')~|~S,S'\in V(G_H)\}.
		\end{align*} Taking maximum on both sides of the above inclusion, we have that $ED(H)$ is less than or equal to the diameter of $G_H$. It is not hard to observe that $\Delta(G_H)=\hat{d^*}(H)$. Since the spectrum of $\mathbf{U}^{\mathbf{L}}(H)$ and the spectrum of $L(G_H)$ are the same, the proof directly follows from~\cite[Theorem~7.4.9]{cvetko}.
	\end{itemize}
\end{proof}
The following theorems provide some spectral bounds on the different types of diameters of a hypergraph defined above.
\begin{thm}
	Let $H$ be a deeply inter uni-connected, simple hypergraph with $e$-index $k$. Then
	$$IUSD(H)\geq EESD(H)\geq ESD(H)\geq \displaystyle \left\lceil \frac{4}{k\cdot\nu_{{k}-1}(H)}\right\rceil.$$	
\end{thm}
\begin{proof}
	Since $H$ is simple and deeply inter-uni-connected, it is clear that $G_H$ is simple and connected. Also, the exact set diameter of $H$ is the same as the diameter of $G_H$. So, the proof follows from~\cite[Theorem~7.5.5]{cvetko} and part~$(iii)$ of Theorem~\ref{chain1}.
\end{proof}

\begin{thm}
	Let $H$ be a deeply connected, simple hypergraph with $e$-index $k$. Then, we have the following:
	\begin{itemize}
		\item[(i)]   $ED(H)\leq SED(H)\leq ESD(H)\leq r-1$, where $r$ is the number of distinct unified Laplacian eigenvalues of $H$.
		\item[(ii)] If $H\neq K_k$, then $$ED(H)\leq SED(H)\leq ESD(H)\leq 1+\left\lfloor\displaystyle \frac{\log(k-1)}{\log\left(\frac{\nu_1(H)+\nu_{k-1}(H)}{\nu_1(H)-\nu_{{k}-1}(H)}\right)}\right\rfloor$$ and
		 $$ED(H)\leq SED(H)\leq ESD(H)\leq 1+\left\lfloor\displaystyle \frac{\cosh^{-1}(k-1)}{\cosh^{-1}\left(\frac{\nu_1(H)+\nu_{{k}-1(H)}}{\nu_1(H)-\nu_{{k}-1(H)}}\right)}\right\rfloor;$$
		\item[(iii)] $ED(H)\leq SED(H)\leq ESD(H)\leq 2 \left\lfloor \sqrt{\frac{2\hat{d^*}(H)}{\nu_{k-1}(H)}}\log_2(k)\right\rfloor$.
	\end{itemize}
\end{thm}
\begin{proof}
	Since $H$ is simple and deeply connected, it is clear that $G_H$ is simple and connected. It is clear that the exact set diameter of $H$ is the same as the diameter of $G_H$. So, the proof follows from~\cite[Proposition~7.5.6 and Theorems~7.5.7, 7.5.8, 7.5.11]{cvetko}, part~$(iii)$ of Theorem \ref{chain2} and Lemma~\ref{Knarg}.
\end{proof}

\begin{notation}\normalfont
	For a hypergraph $H$, we denote $\tau(H):=\underset{e\in E^*(H)}{\bigcup}\tau(e)$.
\end{notation}

\begin{lemma}\label{pathLem}
	Let $H$ be a loopless hypergraph with $e$-index $k$. Then for any $\{T, T'\}\in \tau(H)$, $T$ and $T'$ can consecutively occur in at most $\frac{k^2}{4}$ of any shortest exact path from $S$ to $S'$, where $S, S'\in I(H)$ and $S\neq S'$.
\end{lemma}
\begin{proof}
	Since any $\{T, T'\}\in \tau(H)$ is an edge in $G_H$, and any shortest exact path from $S$ to $S'$ in $H$ is a shortest path from $S$ to $S'$ in $G_H$, the result follows from~\cite[Lemma~4.1]{mohar}:``Let $G$ be a simple graph on $n$ vertices. For each pair $u, v$ of distinct vertices of $G$ choose a shortest path $P_{uv}$ from $u$ to $v$. Then any edge in $G$ belongs to at most $\frac{n^2}{4}$ of the paths $P_{uv}$".
\end{proof}
\begin{thm}\label{EDESD thm}
	Let $H$ be a deeply connected, simple hypergraph with $e$-index $k$. Then 
		\begin{itemize}
		\item[(i)] $UD(H)\geq EED(H)\geq ED(H)\geq\frac{4}{k \nu_{k-1}(H)}-ESD(H)$;
		\item[(ii)] $SUD(H)\geq SEED(H)\geq SED(H)\geq ED(H)\geq\frac{4}{k \nu_{k-1}(H)}-ESD(H).$
	\end{itemize}
	
\end{thm}
\begin{proof}
	\begin{itemize}
		\item [(i)] We begin by proving the last inequality, using a method analogous to the proof of Theorem~4.2 in \cite{mohar}.
		Let $x$ be an eigenvector corresponding to the eigenvalue $\nu_{k-1}(H)$. Then we have		
		\begin{align}\label{E1}
		2k\underset{\{S, S'\}\in \tau(H)}{\sum} (x_S-x_{S'})^2&=\nu_{k-1}(H)\underset{S, S'\in I(H)}{\sum}(x_S-x_{S'})^2\nonumber\\
		&=\nu_{k-1}(H) \left[\underset{(S, S')\in \mathcal{Z}}{\sum}(x_S-x_{S'})^2+\underset{(S, S')\in \mathcal{Z}^c}{\sum}(x_S-x_{S'})^2\right],
		\end{align}
		where $\mathcal{Z}$ is the set of all ordered pairs $(S, S')$ such that $S$ and $S'$ are respectively initial part and terminal part of a shortest exact path in $H$ with $u\in S$, $v\in S'$ for some $u,v\in V(H)$; $\mathcal{Z}^c=I(H)\backslash \mathcal{Z}.$
		
		Consider $(S,S')\in \mathcal{Z}$. Let $P_{SS'}:=S,S_1,S_2,\dots, S_{t-1},S'$ be a corresponding shortest exact path such that $u\in S$, $v\in S'$ for some $u,v\in V(H)$, and let $\mathcal{F}(P_{SS'})=\{\{S,S_1\},\{S_1,S_2\},\dots,$ $\{S_{t-1},S'\}\}$.
		Then, we have
		\begin{align}\label{ED}
		(x_S-x_{S'})^2&=[(x_S-x_{S_1})+(x_{S_1}-x_{S_2})+\dots+(x_{S_{t-1}}-x_{S'})]^2 \nonumber\\ 
		&\leq ed_H(u,v)\underset{\{T,T'\}\in \mathcal{F}(P_{SS'})}{\sum}\delta^2(\{T,T'\})\nonumber\\
		&\leq ED(H) \underset{\{T,T'\}\in \mathcal{F}(P_{SS'})}{\sum}\delta^2(\{T,T'\}),
		\end{align}
		where $\delta^2(\{T,T'\})=(x_T-x_{T'})^2$.
		
		Now consider $(S,S')\in \mathcal{Z}^c$. Since $H$ is deeply connected, there exists a shortest exact path $P'_{SS'}$ from $S$ to $S'$. Then, similar to \eqref{ED}, we get  
		\begin{align}\label{ESD}
		(x_S-x_{S'})^2\leq ESD(H) \underset{\{T,T'\}\in \mathcal{F}(P'_{SS'})}{\sum}\delta^2(\{T,T'\}).
		\end{align}
		
		Applying~\eqref{ED} and \eqref{ESD} in \eqref{E1}, we get
		\begin{align}\label{E2}
		2k\underset{\{S, S'\}\in \tau(H)}{\sum} (x_S-x_{S'})^2\leq &~\nu_{k-1}(H)\underset{(S,S')\in \mathcal{Z}}{\sum} ED(H) \underset{\{T,T'\}\in \mathcal{F}(P_{SS'})}{\sum}\delta^2(\{T,T'\})\nonumber\\
		&+\nu_{k-1}(H)\underset{(S,S')\in \mathcal{Z}^c}{\sum}ESD(H)\underset{\{T,T'\}\in \mathcal{F}(P'_{SS'})}{\sum}\delta^2(\{T,T'\})\nonumber\\
		\leq &~\nu_{k-1}(H)\underset{\{T,T'\}\in \tau(H)}{\sum}\underset{(S,S')\in \mathcal{Z}}{\sum} ED(H)\delta^2(\{T,T'\})\chi_{P_{SS'}}(\{T,T'\})\nonumber\\
		&+\nu_{k-1}(H)\underset{\{T,T'\}\in \tau(H)}{\sum}\underset{(S,S')\in \mathcal{Z}^c}{\sum}ESD(H)\delta^2(\{T,T'\})\chi_{P'_{SS'}}(\{T,T'\}),
		\end{align}
		where $\chi_{P_{SS'}}:\tau(H)\to \{0,1\}$ is the characteristic function of $\mathcal{F}(P_{SS'})$ on $\tau(H)$, i.e., for each $\{T,T'\}\in \tau(H)$,
		
		$$\chi_{P_{SS'}}(\{T,T'\})=
		\begin{cases}
		1, & \text{if}~ \{T,T'\}\in \mathcal{F}(P_{SS'});\\
		0, & \text{otherwise},
		\end{cases}$$
		and similarly, $\chi_{P'_{SS'}}$ is the characteristic function of $\mathcal{F}(P'_{SS'})$ on $\tau(H)$.
		
		Now from Lemma~\ref{pathLem}, for $\{T,T'\}\in \tau(H)$, we have $$\underset{(S,S')\in \mathcal{Z}}{\sum}\chi_{P_{SS'}}(\{T,T'\})\leq \frac{k^2}{2}~\text{and}~\underset{(S,S')\in \mathcal{Z}^c}{\sum}\chi_{P'_{SS'}}(\{T,T'\})\leq \frac{k^2}{2}.$$
		Thus substituting the above inequalities in \eqref{E2}, we have
		\begin{align}\label{E3}
		2k\underset{\{S, S'\}\in \tau(H)}{\sum} (x_S-x_{S'})^2 \leq \frac{\nu_{k-1}(H)k^2}{2}\left(\underset{\{T,T'\}\in \tau(H)}{\sum}\delta^2(\{T,T'\})\right)(ED(H)+ESD(H)).
		\end{align}
	Since $\delta^2(\{T,T'\})=(x_T-x_{T'})^2$ for $\{T,T'\}\in \tau(H)$, \eqref{E3} becomes $$2k \leq \frac{\nu_{k-1}(H)k^2}{2}(ED(H)+ESD(H)).$$
		Hence the last inequality of part~$(i)$ follows.
		
	Since $H$ is deeply connected, it is uni-connected. Therefore, the remaining inequalities follow from part~$(i)$ of Theorem~\ref{chain1} and the last inequality of this part.
		
		\item[(ii)] Since $H$ is deeply connected, it is strong uni-connected. So the first three inequality follows from part~$(ii)$ of Theorem~\ref{chain1}. Also, from part~$(iii)$ of Theorem~\ref{chain2}, $SED(H)\geq ED(H)$. Therefore, part~$(iii)$ follows from the last inequality of part~$(i)$ of this result.
	\end{itemize}
\end{proof}


\subsection{Deep exact components, exact trees and unified Laplacian eigenvalues}
In this subsection, we introduce some more new structures on hypergraphs, and study the interplay between these structures and unified Laplacian spectrum of the hypergraphs. 
\begin{defn}\normalfont
	Let $H$ be a hypergraph. Let $D\subseteq I(H)$ be such that for each $S\in D$ with $|S|>1$ has atleast one neighbor in $D$. The subhypergraph $H'$ of $H$  induced by all the exact paths joining $S$ and $S'$ whose parts belongs to $D$ for all distinct elements $S,S'\in D$ and the vertices corresponding to all the singletons in $D$ is called the \textit{exact subhypergraph of $H$ induced by $D$}.
\end{defn}

Notice that in a graph $G$, the exact subhypergraph of $G$ induced by $D$ is the induced subgraph of $G$ induced by the vertex subset $D\subseteq V(G)$.
\begin{example}\normalfont
	Consider the hypergraph $H$ shown in Figure~\ref{induced}$(a)$. Let $D=\{\{1\},\{4\},\{5\},\{6\},$ $\{2,3\}\}$. Then the exact subhypergraph of $H$ induced by $D$ is as shown in Figure~\ref{induced}$(b)$.
	\begin{figure}[ht]
		\begin{center}
			\includegraphics[scale=1]{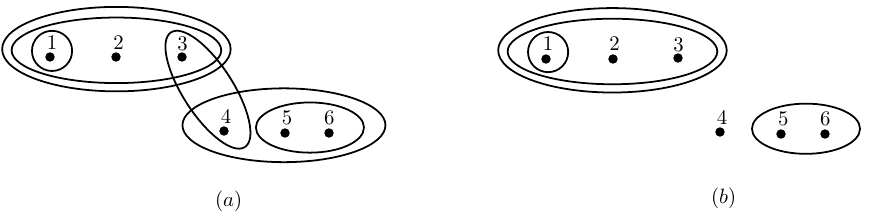}
		\end{center}\caption{(a) The hypergraph $H$; (b) The exact subhypergraph of $H$ induced by $D$}\label{induced}
	\end{figure} 
\end{example}
\begin{defn}\normalfont
	We define a relation $\rho$ on $I(H)$ as follows:
	For $S$, $S'$ in $I(H)$, $S~\rho~S'$ if and only if either $S=S'$ or there exists an exact path joining $S$ and $S'$ in $H$. 
	It is not hard to see that $\rho$ is an equivalence relation on $I(H)$. Let $[S]$ denote the equivalence class of $\rho$ determined by $S\in I(H)$. We call the exact subhypergraph of $H$ induced by $[S]$ as the \textit{deep exact component} (or simply the \textit{$DE$-component})  \textit{of $H$ corresponding to $[S]$}. We call a $DE$-component of $H$ corresponding to $[S]$ as \textit{trivial} if $[S]$ is a singleton and it contains a singleton element of $I(H)$.
\end{defn}
Notice that if $G$ is a graph, then the $DE$-component of $G$ corresponding to the equivalence class $[S]$ is nothing but the component of $G$ whose vertex set is $[S]$.
\begin{defn}\normalfont
	A $DE$-component $H'$ of a hypergraph $H$ is said to have \textit{multiplicity} $k$ if $H'$ is corresponding to exactly $k$ distinct equivalence classes induced by $\rho$.
\end{defn}
If $H$ is a deeply connected hypergraph, then any two distinct elements in $I(H)$ are connected by an exact path. Therefore, $\rho$ induces exactly one equivalence class, which is, $I(H)$, and so $H$ is the only $DE$-component of $H$ with multiplicity one.
\begin{example}\normalfont\label{comp exmp}
	Consider the hypergraph $H$ shown in Figure~\ref{induced}$(a)$.
	Then the equivalence classes corresponding to $\rho$ are the following:
	$[\{1\}]=\{\{1\},\{2,3\}\}$, $[\{2\}]=\{\{2\},\{1,3\}\}$, $[\{3\}]=\{\{3\},\{4\},\{1,2\},$ $\{5,6\}\}$ and
	$[\{5\}]=\{\{5\},\{6\},\{4,5\},\{4,6\}\}$. The subhypergraphs $H_1$, $H_2$, $H_3$ and $H_4$ shown in Figure~\ref{fig4} are the $DE$-components of $H'$ corresponding to $[\{1\}]$, $[\{2\}]$, $[\{3\}]$ and $[\{5\}]$, respectively. Notice that each $DE$-component of $H'$ is with multiplicity one.
\end{example}
\begin{figure}[ht]
	\begin{center}
		\includegraphics[scale=.85]{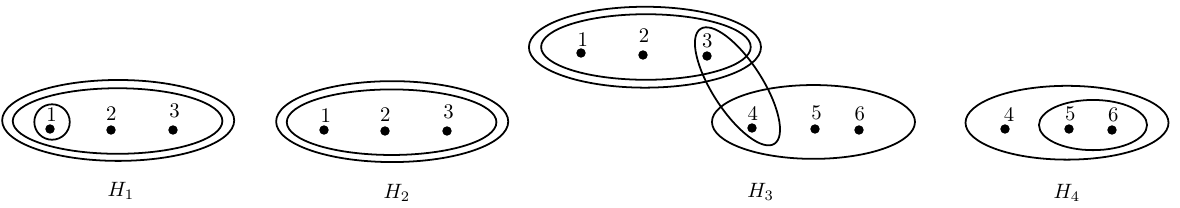}
	\end{center}\caption{The $DE$-components of the hypergraph $H$}\label{fig4}
\end{figure}

\begin{lemma}\label{comp}
Let $H$ be a simple hypergraph. Then there is a one-to-one correspondence between the set of all equivalence classes induced by $\rho$ on $I(H)$, and the set of all components of $G_H$.
	More precisely,
	there is a one-to-one correspondence between the multiset of all $DE$-components of $H$ whose number of occurrences are their multiplicities, and the set of all components of $G_H$.
\end{lemma}
\begin{proof}
	Let $\mathcal{X}$ be the set of all equivalence classes induced by the relation $\rho$ on $I(H)$, and let $\mathcal{Y}$ be the set of all components of $G_H$. 
	
	It can be easily seen that $\mathcal{X}$ is the same as the multiset of $DE$-components of $H$ whose number of occurrences are their multiplicities. Thus to prove this result, it is enough to show that there is a bijection from $\mathcal{X}$ into $\mathcal{Y}$.
	
	Let $S\in I(H)$ and let $[S]\in \mathcal{X}$. Notice that $[S]\subseteq V(G_H)$. Now, we consider the subgraph $G'_H$ of $G_H$ induced by $[S]$. We show that $G'_H$ is a component of $G_H$. Let $S', S''\in V(G'_H)$. Since $V(G'_H)=[S]$, we have $S'\rho S''$. So, there exists an exact path in $H'$ joining $S'$ and $S''$ whose parts belongs to $[S]$. This gives a path joining $S'$ and $S''$ in $G'_H$. Therefore, $G_H'$ is connected. To show $G'_H$ is a component of $G_H$, it is enough to show that there is no $T\in V(G_H)$ with $T\notin V(G'_H)$ such that $T$ is adjacent to any $T'\in V(G'_H)$. Suppose such a $T$ exists. Then $T\rho T'$ for some $T'\in V(G'_H)~(=[S])$ and so $T\in [S]$. This contradict to the fact that $T\notin V(G'_H)$.
	Thus, corresponding to $[S]\in \mathcal{X}$, we get $G_H'\in \mathcal{Y}$ whose vertex set is $[S]$.
	
	Now, we define a map $f:\mathcal{X}\to \mathcal{Y}$ by $f([S])=G_H'$, where $G_H'$ is the subgraph of $G_H$ induced by $[S]$.
	Clearly, the map $f$ is well defined.
	
	We shall show that the map $f$ is $1-1$.
	Let $[S]$, $[T]\in \mathcal{X}$ with $[S]\neq [T]$. Let $G'$ and $G''\in Y$  be the elements associated with $[S]$ and $[T]$ under $f$. Since $[S]\neq[T]$, we have $S\in [S]$ but $S\notin [T]$. So, $S\in V(G')$ but $S\notin V(G'')$. Therefore, $G'$ and $G''$ are two different components of $G_H$, establishing that $f$ is $1-1$.
	
	Next, we show that $f$ is onto. Let $G_H'\in \mathcal{Y}$. Since $G_H'$ is connected, there is a path joining any two of its distinct vertices. Also, there is no vertex $S\in V(G_H)\backslash V(G_H')$ is adjacent to any vertex in $G_H'$. It is evident that $V(G_H')$ is an equivalence class induced by $\rho$ in $I(H)$, i.e., $V(G_H')\in \mathcal{X}$. It is clear that $f(V(G_H'))=G_H'$. Therefore, $f$ is onto. This completes the proof.
\end{proof} 
\begin{thm}\label{thm}
	Let $H$ be a simple hypergraph with $e$-index $k$. Then the number of $DE$-components of $H$ counting with their multiplicities equals $k-r(\mathbf{R}(H))$.
\end{thm}
\begin{proof}
	Notice that $r(\mathbf{R}(H))=r(R(G_H))$. Hence the proof directly follows from  Lemma~\ref{comp} and~\cite[Theorem~2.3]{bapat}: ``If $G$ is a simple graph on $n$ vertices and has $t$ connected components, then $r(R(G))=n-t$".
\end{proof}
\begin{thm}\label{econst}
	Let $H$ be a simple hypergraph. Then the algebraic multiplicity of $0$ as an eigenvalue of $\mathbf{U}^{\mathbf{L}}(H)$ equals the number of $DE$-components of $H$ counting with their multiplicities.	
\end{thm}
\begin{proof}
	Notice that the algebraic multiplicity of the unified Laplacian eigenvalue $0$ of $H$ is the same as the nullity of $\mathbf{U}^{\mathbf{L}}(H)$, which is equal to $|I(H)|-r(\mathbf{U}^{\mathbf{L}}(H))$.
	So, the proof follows from Lemma~\ref{Ob}$(iii)$ and Theorem~\ref{thm}. 
\end{proof}
\begin{cor}
	Let $H$ be a simple hypergraph with $e$-index $k$. Then, $\nu_{k-1}(H)\neq 0$ if and only if $H$ is deeply connected.
\end{cor}
\begin{proof}
	Since $H$ is simple, it is clear that from $(i)$ and $(ii)$ of Lemma~\ref{Ob}, we have $\nu_k=0$.
	If $H$ is deeply connected, then there is exactly one $DE$-component of $H$ which is $H$ itself, and vice versa. Therefore, there is only one equivalence class for $I(H)$ with respect to $\rho$ and so the multiplicity of $H$ is one. Now, from Theorem~\ref{econst}, we have $\nu_{k-1}(H)\neq 0$. 
	Conversely, if $\nu_{k-1}(H)\neq 0$, then $0$ is the algebraic multiplicity of $0$ as an eigenvalue of $\mathbf{U}^{\mathbf{L}}(H)$ is one. Therefore, from Theorem~\ref{econst} $H$ has exactly one $DE$-component with multiplicity one, and so $H$ becomes deeply connected.  
\end{proof}
From the above corollary, $\nu_{k-1}(H)$ characterize the deeply connectedness of the hypergraph $H$. Thereby we define the following
\begin{defn}\normalfont
	Let $H$ be a hypergraph with $e$-index $k$. We refer to the second smallest eigenvalue of $\mathbf{U}^{\mathbf{L}}(H)$ as the \textit{algebraic deep connectivity} or simply the \textit{algebraic $d$-connectivity} of $H$.
\end{defn}
\begin{defn}\normalfont
	A hypergraph $H$ is said to be an \textit{exact tree} or simply an \textit{$e$-tree} if for any pair of distinct elements of $I(H)$, there exists a unique exact path joining them in $H$.
\end{defn}
It can be seen that $H$ is an $e$-tree if and only if $G_H$ is a tree. In graphs, an $e$-tree is a tree. 
\begin{example}\normalfont
	The hypergraph $H$ shown in Figure~\ref{fig5} is an $e$-tree.
	\begin{figure}[ht]\label{tree}
		\begin{center}
			\includegraphics[scale=1]{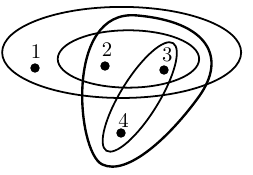}
		\end{center}\caption{The $e$-tree $H$}\label{fig5}
	\end{figure}
\end{example}
\begin{defn}\normalfont
	Let $H$ be a hypergraph. Let $H'$ be a subhypergraph of $H$ such that $I(H')=I(H)$. If there exists $D\subseteq \tau(H')$ such that for any two elements $S$ and $S'\in I(H)$, there exist a unique exact path $(S=)S_0,S_1,\dots, S_{n}(=S')$ joining $S$ and $S'$ in $H'$ with $\{S_{i}, S_{i+1}\}\in D$ for all $i=0,1,\dots,n-1$. Then we call $H'$ as an \textit{exact spanning subhypergraph of $H$}. Also, we refer to $(H',D)$ as an \textit{exact spanning pair} of $H$.
\end{defn}
For a graph $G$, an exact spanning pair $(G',D)$ is a spanning tree $G'$ in $G$, where $D$ is the set of all $2$-partitions of $E(G')$.
\begin{example}\normalfont
	Consider the hypergraph $H$ shown in Figure~\ref{fig1}.
		\begin{figure}[ht]
		\begin{center}
			\includegraphics[scale=1.3]{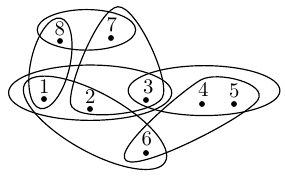}
		\end{center}\caption{The hypergraph $H$}\label{fig1}
	\end{figure} 
Then $H$ is an exact spanning subhypergraph of itself. In addition, $(H, D_1)$ and $(H, D_2)$ are exact spanning pair of $H$, where  $D_1=\tau(H)\backslash\{\{\{1\},\{2,3\}\},\{\{1,2\}, \{3\}\}\}$, and $D_2=\tau(H)\backslash\{\{\{1\},\{2,3\}\},\{\{3\},\{4,5\}\}\}$.
	
	Let $H'$ and $H''$ be hypergraphs obtained from $H$ by deleting the edges $\{1,8\}$ and $\{7,8\}$, respectively. Notice that $I(H)=I(H')=I(H'')$. Then $H'$ and $H''$ are exact spanning subhypergraphs of $H$. In addition, $(H',D')$ and $(H'',D'')$ are exact spanning pairs of $H$, where $D'=\tau(H')\backslash\{\{\{3\},\{4,5\}\}\}$ and $D''=\tau(H'')\backslash\{\{\{3\},\{4,5\}\}\}$.
\end{example}
Let $H$ be a hypergraph. Let $X$ be an eigenvector of $\mathbf{U}^{\mathbf{L}}(H)$. 
For an element $S\in I(H)$, we write $X_S$ to denote the $S$-th entry of $X$.

The following theorem gives some information about the eigenvectors corresponding to the algebraic $d$-connectivity of a hypergraph $H$.
\begin{thm}
	Let $H$ be an $e$-tree. Let $X$ be an eigenvector corresponding to the algebraic $d$-connectivity of $H$. Then exactly one of the following holds:
	\begin{itemize}
		\item No entry of $X$ is zero. In this case, there exists a unique pair of elements of $I(H)$, say $S$, $S'$ such that $S\sim S'$ with $X_S>0$ and $X_{S'}<0$. Further, the entries of $X$ are increasing along
		any exact path in $H$ which starts at $S$ and does not contain $S'$ as its part, while the entries of $X$ are decreasing
		along any exact path in $H$ which starts at $S'$ and does not contain $S$ as its part.
		\item $N_0=\{S\in I(H)~|~X_S=0\}$ is non-empty and any pair of elements in it are joined by an exact path in $H$. Moreover, there is a unique $T\in N_0$ such that $T\sim T'$ with $T'\notin N_0$. 
		The entries of $X$ are either increasing, decreasing,
		or identically $0$ along any exact path in $H$ which starts at $T$.
	\end{itemize}
\end{thm}
\begin{proof}
	Since $H$ is an $e$-tree, $G_H$ is a tree. Also, $\mathbf{U}^{\mathbf{L}}(H)=L(G_H)$. So, the result directly follows from~\cite[Proposition 1.1]{tree}.
\end{proof}
\begin{thm}
	Let $H$ be a non-trivial $e$-tree. Then the multiplicity of any eigenvalue of $H$ cannot exceed the number of elements $S\in I(H)$ with $d^*_H(S)=1$.
\end{thm}
\begin{proof}
	Notice that the set of all elements $S\in I(H)$ with $d^*_H(S)=1$ is nothing but the set of all pendent vertices of $G_H$. It is known that the multiplicity of any eigenvalue of a non-trivial tree $T$ is bounded above by the number of pendent vertices of $T$~(c.f \cite[Excercise 5.2]{cvetko}). So, the result follows.
\end{proof}
\begin{lemma}\label{M-E lem}
	Let $H$ be a simple hypergraph. Then there is a one-to-one correspondence between the set of all exact spanning pairs of $H$, and the set of all spanning subgraphs of $G_H$.
\end{lemma}
\begin{proof}	
	Let $\mathcal{X}$ be the set of all exact spanning pairs $(H',D)$ of $H$, and let $\mathcal{Y}$ be the set of all spanning trees $G$ of $G_H$. 
	
	We show that there is a bijection from $\mathcal{X}$ into $\mathcal{Y}$.
	
	Let $(H',D)\in \mathcal{X}$. Then by the definition of exact spanning pair, given any two elements $S$ and $S'$ of $I(H)$, there exists a unique exact path $(S=)S_0,S_1,\dots, S_{n}(=S')$ in $H'$ with $\{S_{i}, S_{i+1}\}\in D$ for all $i=0,1,\dots,n-1$.
	Corresponding to this we get a unique path joining the vertices $S$ and $S'$ in $G_H$. 
	Let $G$ be the subgraph of $G_H$ induced by these unique paths. Notice that $I(H)=V(G_H)=V(G)$. Therefore, $G\in \mathcal{Y}$. 
	
	Using the above construction, we define a map $f:\mathcal{X}\to \mathcal{Y}$ by $f((H',D))=G$. Clearly, $f$ is well defined.
	
	Now, we shall show that $f$ is $1-1$.
	Let $(H', D')$ and $(H'',D'')\in X$ with $(H', D')\neq(H'',D'')$. Let $G'$ and $G''\in \mathcal{Y}$ be corresponding to $(H',D')$ and $(H'',D'')$ under $f$, respectively. Since $D'\neq D''$, without loss of generality, we assume that $\{S,S'\}\in D'$ but $\{S,S'\}\notin D''$. This implies that $S\cup S'$ is an edge in $G'$ but not an edge in $G''$. Therefore, $G\neq G''$. So $f$ is $1-1$.
	
	Finally, we shall show that $f$ is onto. Let $G\in \mathcal{Y}$. 
	Let $H'$ be the subhypergraph of $H$ induced by all the edges $S\cup S'$, where $SS'\in E(G)$.
	Since $V(G)=I(H)$ and by the definition of spanning tree, corresponding to each unique path in $G$, there exists a unique exact path in $H'$ joining any two elements of $I(H)$. Then clearly, $(H',D')\in X$, where $D'=\{\{S,S'\}~|~SS'\in E(G)\}$. Also, notice that $f((H',D'))=G$. Thus $f$ is onto.
\end{proof}
The Matrix-Tree Theorem for a graph $G$ states that each cofactor of $L(G)$ equals the number of spanning trees of $G$~(c.f.~\cite[Theorem~4.8]{bapat}).
A hypergraph version of the Matrix-Tree Theorem for a graph is given below.
\begin{thm}
Let $H$ be a simple hypergraph. Any cofactor of $\mathbf{U}^{\mathbf{L}}(H)$ is equal to the number of exact spanning pairs of $H$.
\end{thm}
\begin{proof}
Since $H$ is simple, $G_H$ is also simple and vice versa.
	By Lemma~\ref{Ob}$(iv)$, all the cofactors of $\mathbf{U}^{\mathbf{L}}(H)$ are equal.  Thus the result follows from Lemma~\ref{M-E lem}, the fact $\mathbf{U}^{\mathbf{L}}(H)=L(G_H)$ and by applying the Matrix-Tree Theorem to $G_H$.
\end{proof}
We refer to the theorem stated above as the Matrix-Exact Spanning Subhypergraph Theorem for hypergraphs. When applied to graphs, it coincides with the Matrix-Tree Theorem.
\begin{thm}
	If $H$ is a simple hypergraph, then the number exact spanning pairs of $H$ equals $\frac{1}{k}\nu_1(H)\nu_2(H)\cdots\nu_{k-1}(H)$.
\end{thm}
\begin{proof}
	If $H$ is simple, then the eigenvalues of $\mathbf{U}^{\mathbf{L}}(H)$ and $L(G_H)$ are the same. 	
	So, the result follows from Lemma~\ref{M-E lem} and from \cite[Theorem~4.11]{bapat}:``Let $G$ be a simple graph on $n$ vertices. Let the eigenvalues of $L(G)$ be $\nu_1(G)\geq\nu_2(G)\geq\dots\geq\nu_{n}(G)=0$. Then the number of spanning trees of $G$ equals $\frac{1}{n}\nu_1(G)\nu_2(G)\cdots\nu_{n-1}(G)$".
\end{proof}


\section{Unified signless Laplacian matrix of a hypergraph}\label{Lsec5}
 In this section, we introduce the unified signless Laplacian matrix of a hypergraph and study the interplay between the properties of this matrix with the properties of the hypergraph.
\begin{defn}\normalfont
	The \textit{unified signless Laplacian matrix} of a hypergraph $H$, denoted by $\mathbf{U}^{\mathbf{Q}}(H)$, is defined as $\mathbf{U}^{\mathbf{D}}(H)+\mathbf{U}(H)$.
\end{defn}
Let $H$ be a hypergraph with $e$-index $k$. Since $\mathbf{U}^{\mathbf{Q}}(H)$ is a real symmetric matrix, its eigenvalues are all real. We denote them by $\xi_1(H), \xi_2(H),\dots, \xi_k(H)$ and we shall assume that $\xi_1(H)\geq\xi_2(H)\geq\dots\geq\xi_{k}(H)$, where $k$ is the $e$-index of $H$.
The characteristic polynomial of $\mathbf{U}^{\mathbf{Q}}(H)$ is said to be the \emph{unified signless Laplacian characteristic polynomial of $H$}. An eigenvalue of $\mathbf{U}^{\mathbf{Q}}(H)$ is said to be a \emph{unified signless Laplacian eigenvalue of $H$} and the spectrum of $\mathbf{U}^{\mathbf{Q}}(H)$ is said to be the \emph{unified signless Laplacian spectrum of $H$}, or simply the \textit{$\mathbf{U}^{\mathbf{Q}}$-spectrum} of $H$.

\begin{defn}\normalfont
	Let $H$ be a simple hypergraph. The \textit{edge parts incidence matrix} of $H$, denoted by $\mathbf{B}(H)$, is the $0-1$ matrix of order $k\times |\tau(H)|$ whose rows and columns are indexed by the elements of $I(H)$ and $\tau(H)$, respectively. The $(S_i, T_j)-$th entry of $\mathbf{B}(H)$ is $1$ if and only if $S_i\in T_j$. 
\end{defn}

Notice that for a loopless hypergraph $H$,
$\mathbf{U}^{\mathbf{Q}}(H)=Q(G_H)$. If $H$ is simple, then $\mathbf{B}(H)=B(G_H)$. Since $Q(G_H)=B(G_H)B(G_H)^T$, it follows that $\mathbf{U}^{\mathbf{Q}}(H)=\mathbf{B}(H)\mathbf{B}(H)^T$.

Also, if $H$ is a loopless graph, then $\mathbf{U}^{\mathbf{Q}}(H)=Q(H)$, and so the eigenvalues of $Q(H)$ and $\mathbf{U}^{\mathbf{Q}}(H)$ are the same. Thereby, we denote the eigenvalues of these two matrices commonly as $\xi_i(H)$, $i=1,2,\dots,k$. These reveals that the unified signless Laplacian matrix of a hypergraph is a natural generalization of the signless Laplacian matrix of a loopless graph.

The following are some simple observations on the unified signless Laplacian matrix of a hypergraph.
\begin{lemma}\label{lem}
	Let $H$ be a hypergraph with $e$-index $k$. Then, we have the following.
	\begin{itemize}
		\item[(i)] If $H$ is simple, then $\mathbf{U}^{\mathbf{Q}}(H)$ is positive semi-definite;
		\item[(ii)] $\underset{i=1}{\overset{k}{\sum}}\xi_i(H)=vol(H)+\underset{\{v\}\in E(H)}{\sum}m(\{v\})=2\left(\underset{e\in E^*(H)}{\sum}m(e)|\tau(e)|+\underset{\{v\}\in E(H)}{\sum}m(\{v\})\right)$.
	\end{itemize}
	\begin{proof}
		\begin{itemize}
			\item[(i)] Since $H$ is simple, $\mathbf{U}^{\mathbf{Q}}(H)=Q(G_H)$. The result follows from the fact that $Q(G_H)$ is positive semi-definite.
			\item[(ii)] First equality follows from the fact that 
			\begin{align}\label{L eq}
			\underset{i=1}{\overset{k}{\sum}}\xi_i(H)=tr(\mathbf{U}^{\mathbf{Q}}(H))=vol(H)+\underset{\{v\}\in E(H)}{\sum}m(\{v\}).
			\end{align} Second equality follows by substituting the expression for $vol(H)$ mentioned in Lemma~\ref{Ob}(v) in~\eqref{L eq}.
		\end{itemize}
	\end{proof}
\end{lemma}
\begin{thm}
	Let $H$ be a simple hypergraph having no odd exact cycle. Then, $\mathbf{B}(H)$ is totally unimodular.
\end{thm}
\begin{proof}
	Since $H$ has no odd exact cycle, it follows that $G_H$ has no odd cycle and hence $G_H$ is bipartite. Since, $\mathbf{B}(H)$ and $B(G_H)$ have the same spectrum, the proof follows from~\cite[Lemma~2.19]{bapat}: ``If $G$ is a bipartite graph, then $B(G)$ is totally unimodular".
\end{proof}
\begin{thm}
	Let $H$ be a hypergraph with $e$-index $k$. Let $e\in E(H)$ be a non-loop edge such that $I(H-e^r)=I(H)$, where $0<r\leq m(e)$. Then the following hold:\\
	If $|e|=2$, then
	\begin{center} $0\leq\xi_{k}(H-e^r)\leq\xi_{k}(H)\leq \xi_{k-1}(H-e^r)\leq \xi_{k-1}(H)\leq\dots\leq\xi_2(H)\leq\xi_{1}(H-e^r)\leq \xi_{1}(H)$;
	\end{center} otherwise,
	\begin{center} $0\leq\xi_{k}(H-e^r)\leq\xi_{k}(H)\leq \xi_{k-1}(H-e^r)+2r\leq \xi_{k-1}(H)+2r\leq\xi_{k-2}(H-e^r)+4r\leq \xi_{k-2}(H)+4r\leq\dots\leq\xi_2(H)+(k-2)2r\leq\xi_{1}(H-e^r)+(k-1)2r\leq \xi_{1}(H)+(k-1)2r$.
	\end{center}
\end{thm}
\begin{proof}
	Proof is similar to that of Theorem~\ref{L thm}. 
\end{proof}
\begin{thm}
	Let $H$ be a simple hypergraph having no odd exact cycle, then the characteristic polynomial of $\mathbf{U}^{\mathbf{Q}}(H)$ coincides with the
	characteristic polynomial of $\mathbf{U}^{\mathbf{L}}(H)$.
	\end{thm}
 \begin{proof}
 	If $H$ is simple and has no odd exact cycles, $G_H$ is a simple bipartite graph. Since $\mathbf{U}^{\mathbf{Q}}(H)=Q(G_H)$ and $\mathbf{U}^{\mathbf{L}}(H)=L(G_H)$, the proof follows from \cite[Proposition~7.8.4]{cvetko}:“For any simple bipartite graph G, the characteristic polynomial of Q(G)
 	coincides with the characteristic polynomial of L(G)”
 \end{proof}

\begin{thm}\label{thm1}
	Let $H$ be a non-trivial, deeply connected, simple hypergraph with $e$-index $k$. Then, $H$ has no odd exact cycle if and only if $\xi_{k}(H)=0$. Also, in this case, $0$ is a simple unified signless Laplacian eigenvalue. 
\end{thm}
\begin{proof}
	Since $H$ is a non-trivial, deeply connected, simple hypergraph with $e$-index $k$, it turns out that $G_H$ is a non-trivial, connected, simple graph on $k$ vertices.
	Notice that a $DE$-component of $H$ has no odd exact cycle if and only if all the components of $G_H$ corresponding to this $DE$-component are bipartite.  Since $\mathbf{U}^{\mathbf{Q}}(H)=Q(G_H)$, the result directly follows from~\cite[Theorem~7.8.1]{cvetko}: ``Let $G$ be a non-trivial connected simple graph with $n$ vertices. Then $G$ is bipartite if and only if $\xi_n(G)=0$. In this situation, $0$ is a simple eigenvalue". 
\end{proof}
\begin{cor}
	Let $H$ be a hypergraph. Then the multiplicity of the eigenvalue $0$ of $\mathbf{U}^{\mathbf{Q}}(H)$ is equal to the sum of the number of trivial $DE$-components of $H$ and the number of non-trivial equivalence classes $[S]$ induced by $\rho$ on $I(H)$ such that no $T\in [S]$ is a part of an odd exact cycle in $H$. In particular, if $H$ has no odd exact cycle, then the multiplicity of the eigenvalue $0$ of $\mathbf{U}^{\mathbf{Q}}(H)$ is the number of $DE$-components of $H$ counted with their multiplicities.
\end{cor}
\begin{proof}
	If a component of the associated graph $G_H$ of $H$ is non-trivial, then by Theorem~\ref{thm1}, this component has $0$ as its signless Laplacian eigenvalue with multiplicity one if and only if it has no odd exact cycle. i.e., $G_H$ is bipartite. Since the eigenvalues of $\mathbf{U}^{\mathbf{Q}}(H)$ includes all the eigenvalues of each component of $G_H$, it follows that $H$ has $0$ as its signless Laplacian eigenvalue, with the multiplicity equal to the number of components of $G_H$ which are either trivial or bipartite. 
	
	Since for each component of $G_H$, there exists a unique equivalence class induced by $\rho$ on $I(H)$ (c.f.~Lemma~\ref{comp}), let $[S]$ be the equivalence class corresponding to a non-trivial bipartite component $G'$ of $G_H$. Since $G'$ has no odd cycle, it is clear that no $T\in [S]$ is a part of an odd exact cycle in $H$. Therefore, the number of bipartite components of $G_H$ is the same as the number of such equivalence classes induced by $\rho$ on $I(H)$.
	
	Additionally, a $DE$-component of $H$ is trivial if and only if the corresponding component of $G_H$ is trivial. Therefore, the number of  trivial components in $G_H$ is the same as the number of trivial $DE$-components of $H$. 
	
	In particular, no $T\in I(H)$ is a part of an odd exact cycle, as $H$ has no such exact cycles. Thus, the proof follows from the fact that the number of equivalance classes induced by $\rho$ on $I(H)$ is equal to the number of $DE$-components of $H$, counted with their multiplicities (c.f.~Lemma~\ref{comp}).
\end{proof}
\begin{defn}\normalfont
	A hypergraph $H$ is said to be \textit{subset-regular}, if the unified degree of all $S\in I(H)$ are the same.
\end{defn}
It is clear that a subset-regular graph is a regular graph. 

\begin{thm}
	
	Let $H$ be a simple hypergraph with $e$-index $k$ and no included edges. Then, $\xi_1(H)\geq \frac{4}{k}|\tau(H)|$. Equality holds if and only if $H$ is subset-regular. Furthermore, if $H$ is subset-regular, then the degree of any $S\in I(H)$ is $\frac{\xi_1(H)}{2}$, and the number of $DE$-components counted with their multiplicities equals the algebraic multiplicity of $\xi_1(H)$.
\end{thm}
\begin{proof}
	Since, each edge of $G_H$ corresponds to an element in $\tau(H)$, and vice versa, we have $|E(G_H)|= |\tau(H)|$. Notice that $H$ has no included edges and is subset-regular if and only if $G_H$ is regular. Therefore, the proof directly follows from \cite[Theorem~7.8.6]{cvetko}.
\end{proof}
\begin{thm}
	Let $H$ be a hypergraph with $e$-index $k$. Then,
	$$\xi_k(H)\leq\frac{2}{k}\left(\underset{e\in E^*(H)}{\sum}m(e)|\tau(e)|+\underset{\{v\}\in E(H)}{\sum}m(\{v\})\right)\leq\xi_1(H)$$
\end{thm}
\begin{proof}
	Notice that  $\xi_{k}(H)\leq\xi_{i}(H)\leq\xi_{1}(H)$ for all $i=1,2,\dots,k$. Adding these inequalities for $i=1,2,\dots,k$ and from Lemma~\ref{lem}$(ii)$, the result follows.
\end{proof}
	\begin{thm}
		Let $H$ be a simple hypergraph with $e$-index $k$. Let $I(H)=\{S_1,S_2,\dots,S_k\}$. If $d^*_H(S_1)\geq d^*_H(S_2)\geq\dots\geq d^*_H(S_k)$, then $\underset{i=1}{\overset{t}{\sum}}\xi_i(H)\geq\underset{i=1}{\overset{t}{\sum}}d^*_H(S_i)$ for $t=1,2,\dots,k$. The equality holds when $t=k$.	
\end{thm}
\begin{proof}
	Since $\mathbf{U}^{\mathbf{Q}}(H)$ is positive semi-definite and from~\cite[Theorem~1.3.2]{cvetko}:``Let $M$ be a positive
	semi-definite matrix with eigenvalues $\lambda_1 \geq \lambda_2 \geq\dots\geq\lambda_n$. Then $\underset{i=1}{\overset{r}{\sum}}\lambda_i$ is bounded
	below by the sum of the $r$ largest diagonal entries of $M$", we obtain the inequalities of this result. Moreover, if $t=k$, equality holds by Lemma~\ref{lem}$(ii)$.
\end{proof}
\begin{thm}
	For a simple hypergraph $H$, we have $2m^*(H)\leq\xi_1(H)\leq2 \Delta(H)$. In addition, if $H$ is deeply connected, then the equality holds in either place if and only if $H$ is subset-regular and has no included edges.
\end{thm}
\begin{proof}
	It is clear that $H$ is simple if and only if $G_H$ is simple. As the elements of $I(H)$ are the vertices of $G_H$, we have $\delta(G_H)=m^*(H)$. Notice that $H$ is deeply connected if and only if $G_H$ is connected. Moreover, $H$ is subset-regular and has no included edges if and only if $G_H$ is regular. Therefore, the proof follows from~\cite[Proposition~7.8.14]{cvetko}: ``For any simple graph $G$, we have $2\delta(G)\leq\xi_1(G)\leq2 \Delta(G)$. For
	a connected graph $G$, equality holds in either place if and only if $G$ is regular".
\end{proof}
\begin{thm}
	Let $H$ be a simple hypergraph with $e$-index $k$. Then, $\min\{d^*_H(S)+d^*_H(S')\}\leq \xi_1(H)\leq \max \{d^*_H(S)+d^*_H(S')\}$, where the minimum and the maximum runs over all the elements $\{S,S'\}$ of $\tau(H)$.
\end{thm}
\begin{proof}
	Notice that two vertices $S$ and $S'$ in $G_H$ are adjacent if and only if $\{S,S'\}\in \tau(H)$. Since the degree of a vertex $S$ in $G_H$ is the same as $d^*_H(S)$ in $H$, and since $\xi_1(G_H)=\xi_1(H)$, the proof directly follows from \cite[Theorem~7.8.15]{cvetko}: ``Let $G$ be a simple graph on $n$ vertices with vertex degrees $d_1,d_2,\dots, d_n$. Then, $\min\{d_i+d_j\}\leq \xi_1(G)\leq \max \{d_i+d_j\}$, where $(i,j)$ runs over all pairs of adjacent vertices of $G$".
\end{proof}

\begin{thm}
	Let $H$ be a simple hypergraph. Then the following holds:
	\begin{itemize}
		\item [(i)] $\xi_1(H)=0$ if and only if $H$ has no edges;
		\item[(ii)] $\xi_1(H)<4$ if and only if all the $DE$-components of $H$ are exact paths;
		\item[(iii)] For a deeply connected hypergraph $H$, $\xi_1(H)=4$ if and only if $H$ is either a cycle graph or $K_{1,3}$.
	\end{itemize}
\end{thm}
\begin{proof}
	\begin{itemize}   
		\item[(i)] Notice that $H$ has no edge if and only if $G_H$ has no edge.
		\item[(ii)] Let $H'$ be an $DE$-component of $H$ corresponding to the equivalence class $[S]$, where $S\in I(H)$. Then by Lemma~\ref{comp}, $[S]$ admits a component $G'$ in $G_H$ under $f$. Let $[S]=\{S_1,S_2,\dots,S_n\}$. It is clear that if $G'$ is a path in $G_H$ with vertices $S_1, S_2,\dots,S_n$, then $H'$ is an exact path in $H$ with parts $S_1, S_2,\dots,S_n$, and vice versa.

		Thus all the $DE$-components of $H$ are exact paths if and only if all the components of $G_H$ are paths.

		\item[(iii)] If $H$ is deeply connected, then $G_H$ is connected. Also, from parts~$(ii)$ and $(iv)$ of Lemma~\ref{Knarg}, $G_H\cong C_n$ if and only if $H\cong C_n$ and $G_H\cong K_{1,3}$ if and only if $H\cong K_{1,3}$.
	\end{itemize}
	Since $\mathbf{U}^{\mathbf{Q}}(H)=Q(G_H)$, the proof follows from~\cite[Proposition 7.8.16]{cvetko} and from the above facts.
\end{proof}
\begin{thm}
	Let $H$ be a deeply connected, simple hypergraph with $e$-index $k$. Then, $2+2~\cos\frac{\pi}{k}\leq \xi_1(H)\leq 2k-2$.
	The lower bound is attained for $P_k$, and the upper bound is for $K_k$. 
\end{thm}
\begin{proof}
	Since $H$ is deeply connnected and simple, it follows that $G_H$ is connected and simple.	
	The proof follows from parts~$(i)$, $(iii)$ of Lemma~\ref{Knarg} and~\cite[Proposition 7.8.17]{cvetko}: ``Let $G$ be a connected graph on $n$ vertices. Then, $2+2~\cos \frac{\pi}{n}\leq \xi_1(G)\leq 2n-2$.
	The lower bound is attained for $P_n$, and the upper bound is for $K_n$''.
\end{proof}


\section{Unified normalized Laplacian matrix of a hypergraph}\label{Lsec6}
In this section, we introduce the unified normalized Laplacian matrix of a hypergraph and study its spectral properties. Also, we bound its eigenvalues using hypergraph invariants. 
\begin{defn}\normalfont
	Let $H$ be a hypergraph with $e$-index $k$.  \textit{The \textit{unified normalized Laplacian matrix} of $H$}, denoted by $\mathbf{U}^{\bm{\mathcal{L}}}(H)$, is the matrix of order $k$ whose rows and columns are indexed by the elements of $I(H)$ and for all $S_i, S_j\in I(H)$, 
	\begin{center}
		the $(S_i, S_j)^{th}$ entry of $\mathbf{U}^{\bm{\mathcal{L}}}(H)=$
		$\begin{cases}
		1, &\text{if}~i=j,~\text{and}~|S_i|\neq1;\\
		1-\frac{m(\{S_i\})}{d^*_H(S_i)}, &\text{if}~i=j,~|S_i|=1~\text{and}~S_i\in E(H);\\
		-\frac{c}{\sqrt{d^*_H(S_i)d^*_H(S_j)}}, &\text{if}~S_i\overset{c}{\sim} S_j;\\
		0,&\text{otherwise}.
		\end{cases}$
	\end{center}
\end{defn}
It can be seen that if $H$ has no isolated vertices, then $\mathbf{U}^{\bm{\mathcal{L}}}(H)=\mathbf{U}^{\mathbf{D}}(H)^{-\frac{1}{2}}\mathbf{U}^{\mathbf{L}}(H)\mathbf{U}^{\mathbf{D}}(H)^{-\frac{1}{2}}$.

 Since $\mathbf{U}^{\bm{\mathcal{L}}}(H)$ is real symmetric, its eigenvalues all are real. We denote them by $\hat{\nu}_1(H),\hat{\nu}_2(H),$ $\dots,\hat{\nu}_k(H)$ and we shall assume that $\hat{\nu}_1(H)\geq\hat{\nu}_2(H)\geq\dots\geq\hat{\nu}_k(H)$.  The characteristic polynomial of $\mathbf{U}^{\bm{\mathcal{L}}}(H)$ is said to be the \emph{unified normalized Laplacian characteristic polynomial of $H$}. An eigenvalue of $\mathbf{U}^{\bm{\mathcal{L}}}(H)$ is said to be a \emph{unified normalized Laplacian eigenvalue of $H$} and the spectrum of $\mathbf{U}^{\bm{\mathcal{L}}}(H)$ is said to be the \emph{unified normalized Laplacian spectrum of $H$}, or simply \textit{$\mathbf{U}^{\bm{\mathcal{L}}}(H)$-spectrum} of $H$. 

	\begin{note}\normalfont\label{isolated}
	For a loopless hypergraph $H$, $\mathbf{U}^{\bm{\mathcal{L}}}(H)=\hat{L}(G_H)$.
	A hypergraph $H$ has no isolated vertices if and only if $G_H$ has no isolated vertices: For, suppose that $H$ has no isolated vertices. Notice that a vertex of $G_H$ is corresponding to either a vertex of $H$ or a proper subset of an edge of $H$. So, by the construction of $G_H$, the vertices of $G_H$ cannot be isolated. The converse is clear.
	 It is known that for a loopless hypergraph $H$, $\mathbf{U}^{\mathbf{D}}(H)=D(G_H)$ and $\mathbf{U}^{\mathbf{L}}(H)=L(G_H)$. In view of these, for a loopless hypergaph $H$ having no isolated vertices, we have  $$\mathbf{U}^{\bm{\mathcal{L}}}(H)=\mathbf{U}^{\mathbf{D}}(H)^{-\frac{1}{2}}\mathbf{U}^{\mathbf{L}}(H)\mathbf{U}^{\mathbf{D}}(H)^{-\frac{1}{2}}=D(G_H)^{-\frac{1}{2}}L(G_H)D(G_H)^{-\frac{1}{2}}=\hat{L}(G_H).$$
\end{note}
	On the other hand, if $H$ is a loopless graph, then $\mathbf{U}^{\bm{\mathcal{L}}}(H)= \hat{L}(H)$ and so the eigenvalues of $\mathbf{U}^{\bm{\mathcal{L}}}(H)$  and $\hat{L}(H)$ are the same. Thereby, we denote the eigenvalues of these two matrices commonly as $\hat{\nu}_i(H)$, $i=1,2,\dots,k$. These shows that the unified normalized Laplacian matrix of a hypergraph is natural generalization of the normalized Laplacian matrix of a loopless graph.
\begin{thm}
	If $H$ is a hypergraph, then  $\mathbf{U}^{\bm{\mathcal{L}}}(H)$ is singular.
\end{thm}
\begin{proof}
	If $H$ has no edge, then the result is obvious. 
	If $H$ has atleast one edge, then the vector $\mathbf{U}^{\mathbf{D}}(H)^{\frac{1}{2}}J_{k\times 1}$ is non-zero.
	Then, it can be easily seen that $0$ is an eigenvalue of $\mathbf{U}^{\bm{\mathcal{L}}}(H)$ corresponding to the eigenvector $\mathbf{U}^{\mathbf{D}}(H)^{\frac{1}{2}}J_{k\times 1}$ and so $\mathbf{U}^{\bm{\mathcal{L}}}(H)$ is singular. 
\end{proof}
The following theorem counts the arithmetic multiplicity of the zero eigenvalue of $\mathbf{U}^{\bm{\mathcal{L}}}(H)$ using $DE$-components of $H$.
\begin{thm}\label{0 eig}
	The number of $DE$-components of a simple hypergraph $H$ counted with their multiplicities equals the arithmetic multiplicity of $0$ as the eigenvalue of $\mathbf{U}^{\bm{\mathcal{L}}}(H)$.
\end{thm}
\begin{proof}
	As the number of components of $G_H$ equals the number of $DE$-components of $H$ counting with their multiplicities (c.f. Lemma~\ref{comp}), the proof directly follows from~\cite[Theorem~7.7.3]{cvetko}: ``The multiplicity of $0$ as an eigenvalue of $\hat{L}(G)$ equals the number of components of a simple graph $G$". 
\end{proof}

\begin{thm}\label{thm(v)}
	Let $H$ be a simple hypergraph with $e$-index $k~(\geq2)$. Then, we have the following.
	\begin{itemize}
		\item[(i)] $\underset{i=1}{\overset{k}{\sum}}\hat{\nu}_i(H)\leq k$, equality holds if and only if $H$ has no trivial $DE$-component;
		
		\item[(ii)] If $H\neq K_k$, then $\hat{\nu}_{k-1}(H)\leq 1$;
		
		\item[(iii)] If $H$ has no isolated vertices, then $\hat{\nu}_{k-1}(H)\leq \frac{k}{k-1}$ and $\hat{\nu}_{1}(H)\geq \frac{k}{k-1}$; Equality holds if and only if $H=K_k$;
		\item[(iv)] $\hat{\nu}_{1}(H)\leq2$, with equality if and only if there exists a non-trivial equivalence class $[S]$ induced by $\rho$ on $I(H)$ such that no $T\in [S]$ is a part of an odd exact cycle in $H$;
	
		\item[(v)]
		$\underset{i=1}{\overset{k}{\sum}}\hat{\nu}_i(H)=k-t$, where $t$ is the number of trivial $DE$-components of $H$.
	\end{itemize}
	\begin{proof}
		We observe the following:
			If the $e$-index of $H$ is atleast $2$, then $G_H$ has atleast $2$ vertices. Also, $H$ is simple if and only if $G_H$ is simple.
				
		Since $\mathbf{U}^{\bm{\mathcal{L}}}(H)=\hat{L}(G_H)$, the results $(i)$-$(iv)$ follow from \cite[Theorem~$7.7.2$]{cvetko} and from the respective  arguments given below.
		
		\begin{itemize}
			\item[(i)] From Lemma~\ref{Knarg}$(v)$, it is known that $H$ has no isolated vertices if and only if $G_H$ has no isolated vertices. 
			 Also, notice that in a simple hypergraph, the isolated vertices are the trivial $DE$-components of $H$.
			 			
			\item[(ii)]and (iii) follows from Lemma~\ref{Knarg}$(iii)$.
			\item[(iv)] As in the proof of Lemma~\ref{comp}, if $H$ has a non-trivial $DE$-component $H'$ corresponding to $[S]$, then $[S]$ admits a non-trivial component $G_H'$ in $G_H$ under $f$. Notice that $[S]$ has no $T$ as a part of any odd exact cycle if and only if $G_H'$ has no odd cycle and so it is bipartite. Thus $H$ has a non-trivial $DE$-component which has no odd exact cycle if and only if $G_H$ has a non-trivial bipartite component. 
			\item[(v)]Suppose $H$ has $t$ trivial $DE$-components. Then $H$ has $t$ isolated vertices. Let $H'$ be the subhypergraph of $H$ obtained by deleting all its isolated vertices. Then, $H'$ has no trivial $DE$-component and has $e$-index $k-t$. Therefore, by part~$(i)$, $\underset{i=1}{\overset{k-t}{\sum}}\hat{\nu}_i(H')=k-t$. Since the trace of $\mathbf{U}^{\bm{\mathcal{L}}}(H)$ and $\mathbf{U}^{\bm{\mathcal{L}}}(H')$ are the same, we have $\underset{i=1}{\overset{k-t}{\sum}}\hat{\nu}_i(H')=\underset{i=1}{\overset{k}{\sum}}\hat{\nu}_i(H)$.
		\end{itemize}
	\end{proof}
\end{thm}
\begin{cor}
	Let $H$ be a hypergraph. Then the multiplicity of the eigenvalue $2$ of $\mathbf{U}^{\bm{\mathcal{L}}}(H)$ is equal to the number of non-trivial equivalence classes $[S]$ induced by $\rho$ on $I(H)$ such that no $T\in [S]$ is a part of an odd exact cycle in $H$. In particular, if $H$ has no odd exact cycle, then the number of non-trivial $DE$-components of $H$ counted with their multiplicities equals the arithmetic multiplicity of $2$ as the eigenvalue of $\mathbf{U}^{\bm{\mathcal{L}}}(H)$.
\end{cor}
\begin{proof}
	Let $[S]$ be a non-trivial equivalence class induced by $\rho$ on $I(H)$ such that no $T\in [S]$ is part of an odd exact cycle in $H$. Corresponding to $[S]$, there exists a component $G'$ of $G_H$ (c.f.~Lemma~\ref{comp}). Also, notice that $V(G')(=[S])$ is an equivalence class induced by $\rho$ on $V(G_H)$. Thus from Theorem~\ref{thm(v)}$(iv)$, $G'$ has $2$ as a normalized Laplacian eigenvalue. Since the eigenvalues of $\mathbf{U}^{\bm{\mathcal{L}}}(H)$ contains all the eigenvalues of each component of $G_H$, it follows that $H$ has $2$ as its normalized Laplacian eigenvalue with multiplicity equals the number of equivalence classes induced by $\rho$ on $I(H)$ such that no $T\in [S]$ is part of an odd exact cycle in $H$. 
	
	In particular, if $H$ has no odd exact cycle, no $T\in I(H)$ is part of an odd exact cycle. Then proof follows from the fact that the number of equivalance classes induced by $\rho$ on $I(H)$ equals the number of $DE$-components of $H$, counted with their multiplicities. 
\end{proof}

\begin{thm}
	A simple hypergraph $H$ with $e$-index $k$ has no odd exact cycle if and only if $\hat{\nu}_{1}(H)=2$ with the same multiplicity as $\hat{\nu}_{k}(H)$.
\end{thm}
\begin{proof}
	It is clear that $H$ has no odd exact cycle if and only if $G_H$ has no odd cycle, i.e., if and only if $G_H$ is bipartite.	
	So, the result follows from~\cite[Corollary 7.7.4]{cvetko}: ``A simple graph $G$ is bipartite on $n$ vertices if and only if the eigenvalue $\hat{\nu}_{1}(G)=2$ with the same multiplicity as $\hat{\nu}_{n}(G)$".
\end{proof}
\begin{thm}
	Let $H$ be a simple deeply connected hypergraph. If each edge of $H$ belongs to at least $r$ exact cycles of length $3$, then \[\hat{\nu}_{1}(H)\leq \underset{S\in I(H)}{\max}\bigg\{1+\frac{1}{2d^*_H(S)}\left(\sqrt{4d^*_H(S)(t(S)-r)+r^2}-r\right)\bigg\},\] where $t(S)=\underset{S'\in N_H(S)}{\sum}d^*_H(S')$. The equality holds for $K_{r+2}$.
\end{thm}
\begin{proof}
	Since $H$ is simple and deeply connected, it follows that $G_H$ is simple and connected. Now, the result follows from~\cite[Theorem~$3.2$]{NLbound}, Lemma~\ref{Knarg}$(iii)$ and from the fact that an exact cycle of length $3$ in $H$ becomes a triangle in $G_H$. 
\end{proof}
\begin{thm}
	Let $H$ be a simple deeply connected hypergraph with $e$-index $k$. Then we have the following.
	\begin{itemize}
		\item [(i)] $\displaystyle\hat{\nu}_{k-1}(H)\geq \frac{1}{ESD(H)vol(H)}$;
		\item [(ii)] If $ESD(H)\geq 4$, then  $\displaystyle\hat{\nu}_{k-1}(H)\leq 1-2\frac{\sqrt{\hat{d^*}(H)-1}}{\hat{d^*}(H)}\left(1-\frac{2}{ESD(H)}\right)+\frac{2}{ESD(H)}$.
	\end{itemize} 
\end{thm}
\begin{proof}
		Since $H$ is simple and deeply connected, $G_H$ is simple and connected. Also, notice that $ESD(H)$ is the same as the diameter of $G_H$.
Part~$(i)$ directly follows from~\cite[Lemma~1.9]{chung bk}.
	 It is not hard to observe that $\Delta(G_H)=\hat{d^*}(H)$. Thus, part~$(ii)$ follows from~\cite[Lemma~1.14]{chung bk}.
\end{proof}
\begin{thm}
	Let $H$ be a simple deeply connected hypergraph. Then \[\hat{\nu}_{1}(H)\geq \displaystyle\frac{2|\tau(H)|}{2|\tau(H)|-\Delta(H)}.\]
\end{thm}
\begin{proof}
	It is clear that the number of edges in $G_H$ equals $|\tau(H)|$.
	Notice that $\Delta(H)=\Delta(G_H)$. So, the result follows from~\cite[Corollary 3.6]{NLbound}: ``Let $G$ be a simple connected graph with $m$ edges. Then $\hat{\nu}_{1}(G)\geq \frac{2m}{2m-\Delta(G)}$". 
\end{proof}
\begin{thm} Let $H$ be a simple deeply connected hypergraph with $e$-index $k$.
	\begin{itemize}
		\item[(i)] 	If $H\neq K_{k}$, then 
		$\hat{\nu}_{k-1}(H)=1$ if and only if $H$ is a complete multipartite graph;		
			
		\item[(ii)]  $\hat{\nu}_{2}(H)\geq 1$; the equality holds if and only if $H$ is a complete bipartite graph.
	\end{itemize}
\end{thm}
\begin{proof}
	Since $H$ is a simple and deeply connected hypergraph, $G_H$ is a simple connected graph.
	
		Part~$(i)$ follows from parts~$(iii)$,  $(iv)$ of Lemma~\ref{Knarg} and from~\cite[Theorem~3.11]{NLbound}: ``Let $G(\neq K_n)$ be a simple connected graph on $n$ vertices. Then $\hat{\nu}_{n-1}(G)=1$ if and only if $G$ is a complete multipartite graph".
		
		Part~$(ii)$ follows from Lemma~\ref{Knarg}$(iv)$ and from \cite[Theorem~3.12]{NLbound}: ``Let $G$ be a connected graph on $n$ vertices. Then $\hat{\nu}_{2}(G)\geq 1$. The equality holds if
		and only if $G$ is a complete bipartite graph".
\end{proof}


\subsection{Unified Cheeger constant, exact set distance and unified normalized Laplacian eigenvalues}
In this subsection, we introduce the unified Cheeger constant of a hypergraph, and the exact set distance between two subsets of $I(H)$; and obtain some spectral bounds on these invariants.
\begin{defn}\normalfont
	Let $H$ be a simple hypergraph. Let $\mathcal{X}$ and $\mathcal{Y}$ be non-empty subsets of $I(H)$. We denote $\hat{E}(\mathcal{X},\mathcal{Y})=\{\{S, S'\}\in \tau(H)\colon S\in \mathcal{X}, S'\in \mathcal{Y}~\text{or}~S\in \mathcal{Y}, S'\in \mathcal{X}\}$. Let $\mathcal{X}^c:=I(H)\backslash \mathcal{X}$. Then we denote $\hat{E}(\mathcal{X},\mathcal{X}^c)$ simply by $\hat{E}(\mathcal{X})$. 
	
	For a non-empty proper subset $\mathcal{X}$ of $I(H)$, we define $$uc(\mathcal{X})=\displaystyle\frac{|\hat{E}(\mathcal{X})|}{\min \{vol_H(\mathcal{X}), vol_H(\mathcal{X}^c)\}}.$$ 
	
	 We define the \textit{unified Cheeger constant} of $H$ as $$uc(H)=\underset{\mathcal{X}\subset I(H)}{\min}uc(\mathcal{X}).$$ 
\end{defn}
Notice that each edge in $\hat{E}(\mathcal{X})$ of $H$ contributes to the values $vol_H(\mathcal{X})$ and $vol_H(\mathcal{X}^c)$. So we have $\min \{vol_H(\mathcal{X}), vol_H(\mathcal{X}^c)\}\geq |\hat{E}(\mathcal{X})|$. Therefore, for any simple hypergraph, $uc(X)\leq 1$ and so, $uc(H)\leq 1$.
\begin{thm}
	Let $H$ be a simple hypergraph with $e$-index $k$ and let $\nu^*(H)=\underset{i\neq k}{\max}|1-\hat{\nu}_i(H)|$.
	Let $\mathcal{X}$ and $\mathcal{Y}$ be two non-empty subsets of $I(H)$.  Then we have the following.
	\begin{itemize}
		\item[(i)] $\displaystyle\left||\hat{E}(\mathcal{X},\mathcal{Y})|-\frac{vol_H(\mathcal{X})vol_H(\mathcal{Y})}{vol(H)}\right|\leq\nu^*(H)\sqrt{vol_H(\mathcal{X})vol_H(\mathcal{Y})}$;
\item[(ii)] $\displaystyle\left||\hat{E}(\mathcal{X},\mathcal{Y})|-\frac{vol_H(\mathcal{X})vol_H(\mathcal{Y})}{vol(H)}\right|\leq\nu^*(H)\frac{\sqrt{vol_H(\mathcal{X})vol_H(\mathcal{Y})vol_H(\mathcal{X}^c)vol_H(\mathcal{Y}^c)}}{vol(H)}$;
\item[(iii)] $\displaystyle\left||\hat{E}(\mathcal{X},\mathcal{X})|-\frac{vol_H(\mathcal{X})^2}{vol(H)}\right|\leq\nu^*(H)\frac{\sqrt{vol_H(\mathcal{X})vol_H(\mathcal{X}^c)}}{vol(H)}\leq\nu^*(H)vol_H(\mathcal{X})$;
	\end{itemize}
\end{thm}
\begin{proof}
	Notice that for any $\mathcal{S}\subset I(H)$,  $vol_H(\mathcal{S})=vol(\mathcal{S})$. Also, $|\hat{E}(\mathcal{X},\mathcal{Y})|=|E(\mathcal{X},\mathcal{Y})|$. Since the spectrum of $\mathbf{U}^{\bm{\mathcal{L}}}(H)$ and $\hat{L}(G_H)$ are the same, the parts (i), (ii) and (iii) follows from Theorems~5.1, 5.2 and Corollary 5.3 of~\cite{chung bk}, respectively.
\end{proof}

\begin{thm}
	Let $H$ be a simple, deeply connected hypergraph. Then for any non-empty subset $\mathcal{S}\subseteq I(H)$, we have
	$$\hat{\nu}_1(H)\geq\displaystyle\frac{2|\tau(H)|}{vol_H(\mathcal{S})(2\tau(H)-vol_H(\mathcal{S}))}.$$
\end{thm}
\begin{proof}
		The number of edges in $G_H$ equals $|\tau(H)|$.
	Also, notice that $vol_H(\mathcal{S})=vol(\mathcal{S})$. So, the result directly follows from~\cite[Theorem~3.5]{NLbound}.
\end{proof}
\begin{thm}
	Let $H$ be a simple deeply connected hypergraph with $e$-index $k$. Then we have the following.
	\begin{itemize}
		\item[(i)] $uc(H)^2<2\hat{\nu}_{k-1}(H)\leq 4 uc(H)$;
		\item[(ii)] $\hat{\nu}_{k-1}(H)>1-\sqrt{1-uc(H)^2}$.
	\end{itemize}
\end{thm}
\begin{proof}
	Since $uc(H)=h(G_H)$ and $\mathbf{U}^{\bm{\mathcal{L}}}(H)=\hat{L}(G_H)$, the results (i) and (ii) directly follow from Lemma~2.1, Theorems~2.2 and~2.3 of~\cite{chung bk}.
\end{proof}
\begin{defn}\normalfont
	Let $\mathcal{X}$ and $\mathcal{Y}$ be two non-empty subsets of $I(H)$. We define the {\it exact set distance between $\mathcal{X}$ and $\mathcal{Y}$} as the minimum of the exact set distance between each $S\in \mathcal{X}$ and $S'\in\mathcal{Y}$. We denote it by $esd_H(\mathcal{X}, \mathcal{Y})$, i.e.,
	$esd_H(\mathcal{X}, \mathcal{Y})=\min\{esd_H(S,S')\colon S\in \mathcal{X}, S'\in\mathcal{Y}\}$.
\end{defn}
\begin{thm}
	Let $H$ be a simple hypergraph with $e$-index $k$. 
	\begin{itemize} 
		\item [(1) ] Let $\mathcal{X}$ and $\mathcal{Y}$ be proper subsets of $I(H)$.
		\begin{itemize}
			\item[(i)]
			If $H\ncong K_k$, then $esd_H(\mathcal{X}, \mathcal{Y})\leq\displaystyle\left\lceil \frac{\log\sqrt{\frac{vol_H(\mathcal{X}^c)vol_H(\mathcal{Y}^c)}{vol_H(\mathcal{X})vol_H(\mathcal{Y})}}}{\log \frac{\hat{\nu}_{1}(H)+\hat{\nu}_{k-1}(H)}{\hat{\nu}_{1}(H)-\hat{\nu}_{k-1}(H)}}\right\rceil$;\\
			\item[(ii)]
			If $H\ncong K_k$, then $esd_H(\mathcal{X}, \mathcal{Y})\leq\displaystyle\left\lceil \frac{\cosh^{-1}\sqrt{\frac{vol_H(\mathcal{X}^c)vol_H(\mathcal{Y}^c)}{vol_H(\mathcal{X})vol_H(\mathcal{Y})}}}{\cosh^{-1} \frac{\hat{\nu}_{1}(H)+\hat{\nu}_{k-1}(H)}{\hat{\nu}_{1}(H)-\hat{\nu}_{k-1}(H)}}\right\rceil$.\\
		\end{itemize}
		
		\item [(2)]	 Let $\mathcal{X}_i$ be proper subsets of $I(H)$, $i=1,2,\dots,t$. 
		\begin{itemize}
			\item[(i)]	
			If $H\ncong K_k$, then $\underset{i\neq j}{\min}~esd_H(\mathcal{X}_i,\mathcal{X}_j)\leq	\underset{i\neq j}{\max}\displaystyle\left\lceil \frac{\log\sqrt{\frac{vol_H(\mathcal{X}_i^c)vol_H(\mathcal{X}_j^c)}{vol_H(\mathcal{X}_i)vol_H(\mathcal{X}_j)}}}{\log \frac{1}{1-\hat{\nu}_{k-t+1}(H)}}\right\rceil,$  whenever $1-\hat{\nu}_{k-t+1}(H)\geq \hat{\nu}_1(H)-1$;
			\item[(ii)] $\underset{i\neq j}{\min}~esd_H(\mathcal{X}_i,\mathcal{X}_j)\leq	\underset{i\neq j}{\max}\displaystyle\left\lceil \frac{\log\sqrt{\frac{vol_H(\mathcal{X}_i^c)vol_H(\mathcal{X}_j^c)}{vol_H(\mathcal{X}_i)vol_H(\mathcal{X}_j)}}}{\log \frac{\hat{\nu}_1(H)+\hat{\nu}_{k-t+1}(H)}{\hat{\nu}_1(H)-\hat{\nu}_{k-t+1}(H)}}\right\rceil$, if $\hat{\nu}_1(H)\neq\hat{\nu}_{k-t+1}(H)$;
			\item[(iii)] $\underset{i\neq j}{\min}~esd_H(\mathcal{X}_i,\mathcal{X}_j)\leq	\underset{1\leq j\leq t}{\min}\underset{i\neq j}{\max}\displaystyle\left\lceil \frac{\log\sqrt{\frac{vol_H(\mathcal{X}_i^c)vol_H(\mathcal{X}_j^c)}{vol_H(\mathcal{X}_i)vol_H(\mathcal{X}_j)}}}{\log \frac{\hat{\nu}_{j+1}(H)+\hat{\nu}_{k-t+j-1}(H)}{\hat{\nu}_{j+1}(H)-\hat{\nu}_{k-t+j-1}(H)}}\right\rceil$, if $\hat{\nu}_{j+1}(H)\neq\hat{\nu}_{k-t+j-1}(H)$;
		\end{itemize}
	\end{itemize}
\end{thm}
\begin{proof}
	Since $esd_H(S,S')$ is the same as the distance between two vertices $S$ and $S'$ in $G_H$, it can be easily seen that $esd_H(\mathcal{X}, \mathcal{Y})$ is the same as the distance between the vertex subsets $\mathcal{X}$ and $\mathcal{Y}$ of $G_H$.
	Also, for any $\mathcal{S}\subset I(H)$, we have $vol_H(\mathcal{S})=vol(\mathcal{S})$ and $\mathbf{U}^{\bm{\mathcal{L}}}(H)=\hat{L}(G_H)$. Thus (i)-(ii) of part~1, and (i)-(iii) of part~2 follow from Lemma~\ref{Knarg}$(iii)$ and Theorems~3.1, 3.3, 3.10, 3.11 and 3.12 of~\cite{chung bk}.
\end{proof}


\section{Cospectral hypergraphs}\label{Lsec7}

In this section, we present some facts about cospectral hypergraphs with respect to the four matrices considered in this paper.

 Hypergraphs that have the same spectrum of an associated matrix $M$ are referred to as cospectral hypergraphs with respect to $M$, or simply $M$-cospectral hypergraphs.
 A hypergraph $H$ is said to be determined by its spectrum with respect to $M$ if there is no other non-isomorphic hypergraph $M$-cospectral with $H$.

 The cospectrality of graphs with respect to the adjacency matrix, Laplacian matrix,  signless Laplacian matrix, and  normalized Laplacian matrix has been studied in the literature (see, for example,~\cite{butler2015},\cite{carva2017},\cite{haemers2004}).   In addition, graphs that are determined by the spectrum of each of these four matrices are found
  in the literature  (see, for example,~\cite{berman2018},\cite{liu2011},\cite{Van Dam}).
 
 Since a loopless hypergraph and its associated graph have the same unified spectrum, unified Laplacian spectrum, unified signless Laplacian spectrum and unified normalized Laplacian  spectrum, the problem of finding the cospectral hypergraphs with respect to the unified matrix (resp. unified Laplacian matrix, unified signless Laplacian matrix and unified normalized Laplacian matrix) reduces to addressing the following question: ``\textit{For a given family $\mathcal{G}$ of cospectral graphs having at least one loopless graph with respect to the adjacency matrix (resp. Laplacian matrix,  signless Laplacian matrix and  normalized Laplacian matrix), what is the family of hypergraphs whose associated graphs belongs to $\mathcal{G}$?}"

The number of cospectral families of all loopless graphs  with respect to the adjacency matrix is the same as the number of cospectral families of all loopless hypergraphs with respect to the unified matrix. Moreover, each such cospectral family $\mathcal{G}$ of loopless graphs must be contained in the cospectral family $\mathcal{H}$ of loopless hypergraphs whose associated graphs belong to $\mathcal{G}$. However, $\mathcal{G}$ and $\mathcal{H}$ are the same if and only if for each $G\in \mathcal{G}$, there is no  hypergraph $H \in \mathcal{H}$, other than $G$ whose  associated graph is $G$. For instance, as per~\cite{Van Dam}, the graphs $P_n$, $C_n$, $K_n$, and $K_{m,m}$ are uniquely determined with respect to the adjacency spectra. Consequently, the cospectral family $\mathcal{G}$ of each of these graphs must consist of only that graph.  With these information and from Lemmas~\ref{Knarg}, we conclude that $\mathcal{H}= \mathcal{G}$ for each of these graphs.

Since,  the unified matrix of a loopless hypergraph and its associated graph are the same, and the unified matrix and the adjacency matrix of a loopless graph are the same, we conclude that any loopless hypergraph which is  determined by the spectrum of its unified matrix must, in fact, be a graph and that is  determined by the adjacency spectrum. Conversely, a graph which is  determined by its adjacency spectrum may not necessarily be  determined by its unified spectrum. For instance, the graph $G$, which is the disjoint union of $3$ copies of $K_2$, is  determined with respect to the adjacency matrix (c.f.~\cite{Van Dam}); but not with respect to the unified matrix, since $G$ and the complete hypergraph on $3$ vertices have the same unified spectrum. However, the converse is true if the following condition is satisfied: \textit{Let $G$ be a graph which is determined by its adjacency spectrum. Then $G$ is  determined by its unified spectrum if and only if there is no hypergraph, other than $G$ itself, whose associated graph is $G$}. 

The above information regarding the unified matrix of a hypergraph also applies to the unified Laplacian matrix, the unified signless Laplacian matrix and the unified normalized Laplacian matrix.


\section*{Conclusion}
The unified Laplacian matrix, unified signless Laplacian matrix and unified normalized Laplacian matrix associated with a hypergraph provides a unified approach for linking spectral hypergraph theory with  the spectra of the adjacency matrices of graphs. In this context,  we introduced certain hypergraph structures and invariants, such as 
edge exact connectedness, inter-uni-connectedness, uni-connectedness, strong exactly connectedness, strong edge exact connectedness, strong uni-connectedness, deeply connectedness, deeply edge exact connectedness, deeply inter-uni-connectedness,
 edge exact distance, internal unified distance, unified distance, strong exact distance, strong edge exact distance, strong unified distance, exact set distance, edge exact set distance, internal unified set distance,
edge exact diameter, internal unified diameter, unified diameter, strong exact diameter, strong edge exact diameter, strong unified diameter, exact set diameter, edge exact set diameter, internal unified set diameter,
  deep exact components, algebraic d-connectivity, exact tree, unified Cheeger constant, and related them to the eigenvalues of these matrices. Although the relationships between these structures and invariants and the spectrum of the unified matrix of the hypergraph have been established, further research is needed to explore their properties from a non-spectral hypergraph theoretical viewpoint. Moreover, this approach allows for the extension of various results and properties of graphs, typically expressed using Laplacian matrix, signless Laplacian matrix and  normalized Laplacian matrix, respectively, to hypergraphs using the unified Laplacian matrix, unified signless Laplacian matrix and unified normalized Laplacian matrix.

Since the unified Laplacian matrix, unified signless Laplacian matrix and unified normalized Laplacian matrix of a hypergraph are natural generalization of Laplacian matrix, signless Laplacian matrix and normalized Laplacian matrix, respectively of a loopless graph, the results we established in this paper generalize the corresponding results in graphs.

\end{document}